\documentclass[11pt,reqno]{amsart} 

\usepackage[utf8]{inputenc}
\usepackage[a4paper]{geometry}

\usepackage{amsmath}
\usepackage{amsfonts}
\usepackage{amssymb}
\usepackage{amsthm}

\usepackage{natbib}
\usepackage{color}
\usepackage[colorlinks=true,linkcolor=blue,citecolor=blue,pdfborder={0 0 0}]{hyperref}

\usepackage{dsfont}
\usepackage{bm}

\usepackage{graphicx}
\usepackage{enumerate}

\newtheorem{theorem}{Theorem}[section]
\newtheorem{proposition}[theorem]{Proposition}
\newtheorem{lemma}[theorem]{Lemma}

\theoremstyle{definition}

\theoremstyle{remark}
\newtheorem{remark}[theorem]{Remark}

\numberwithin{equation}{section}

\newcommand{\dto}{\rightsquigarrow}

\newcommand{\bb}[1]{\mathds{#1}}

\newcommand{\reals}{\mathbb{R}}
\newcommand{\R}{\reals}

\newcommand\Prob{\mathbb{P}}    
\newcommand{\Probn}{\operatorname{\Prob}_n}

\newcommand{\cd}{{\Theta}}             
\newcommand{\nd}{{\Theta_{0}}}	
\newcommand{\nnd}{{\Theta_{0,-}}}	
\newcommand{\pnd}{{\Theta_{0,+}}}	

\newcommand{\Gb}{\operatorname{\mathbb{G}}}

\newcommand{\Go}{\bar{G}}
\newcommand{\ello}{\bar{\ell}}

\newcommand{\Ac}{\mathcal{A}}

 \newcommand{\Mcc}{\mathcal{M}}
\newcommand{\Xc}{\mathcal{X}}

\newcommand{\eps}{\varepsilon}
\newcommand{\abs}[1]{\left\vert{#1}\right\vert}
\newcommand{\norm}[1]{\lVert{#1}\rVert}
\newcommand{\ind}{\operatorname{\bb{1}}}
\newcommand{\diff}{\mathrm{d}}
\newcommand{\point}{\,\cdot\,}

\begin{document}

\title[On the MLE for the GEV]{On the maximum likelihood estimator for the Generalized Extreme-Value distribution}

\author{Axel B\"ucher}
\address{Ruhr-Universit\"at Bochum, Fakult\"at f\"ur Mathematik, Universit\"atsstr.\ 150, 44780 Bochum, Germany}
\email{axel.buecher@rub.de}

\author{Johan Segers}
\address{Universit\'{e} catholique de Louvain, Institut de Statistique, Biostatistique et Sciences Actuarielles, Voie du Roman Pays~20, B-1348 Louvain-la-Neuve, Belgium}
\email{johan.segers@uclouvain.be}

\date{\today}

\begin{abstract}
The vanilla method in univariate extreme-value theory consists of fitting the three-parameter Generalized Extreme-Value (GEV) distribution to a sample of block maxima. 
Despite claims to the contrary,
the asymptotic normality of the maximum likelihood estimator has never been established. 
In this paper, a formal proof is given using a general result on the maximum likelihood estimator for parametric families that are differentiable in quadratic mean but whose supports depend on the parameter. An interesting side result concerns the (lack of) differentiability in quadratic mean of the GEV family.

\bigskip


\end{abstract}

\maketitle


\section{Introduction}

Asymptotic normality of maximum likelihood estimators in regular parametric models is a classic subject, but when the support of the distribution depends on the parameter, the mathematics are not routine. The family of univariate, three-parameter Generalized Extreme-Value distributions is a case in point. Its cumulative distribution function $G_\theta$ at a parameter vector $\theta = (\gamma, \mu, \sigma) \in \R \times \R \times (0, \infty)$ is given by
\begin{equation}
\label{eq:GEV}
  G_{\gamma,  \mu, \sigma} (x) 
  = 
  G_{\gamma, 0, 1}( (x - \mu)/\sigma ),
  \qquad x \in \reals,
\end{equation}
where
\begin{align*}
  G_{\gamma, 0, 1}(z) 
  =
  \begin{cases}
    \exp(-e^{-z}) 
      & \text{if $\gamma = 0$}, \\
    \exp\left\{ -(1+\gamma z)^{-1/\gamma} \right\}  
      & \text{if $\gamma \ne 0$ and $1 + \gamma z > 0$}, \\
    0 & \text{if $\gamma > 0$ and $z \le -1/\gamma$}, \\
    1 & \text{if $\gamma < 0$ and $z \ge -1/\gamma$}.
  \end{cases}
\end{align*}
The support of $G_\theta$ is an interval, $S_\theta = \{x \in \reals : \sigma + \gamma (x - \mu) > 0 \}$, whose endpoints depend on $\theta$. The above parameterization, due to \cite{vonmises:1936} and \citet{jenkinson:1955}, generalizes and unifies the Fr\'echet/Gumbel/Weibull trichotomy in \citet{fisher+t:1928} across all signs of the shape parameter $\gamma$.

Fitting a generalized extreme-value distribution to a sample of annual maxima is the earliest statistical method in extreme-value theory. Various inference procedures have been explored, including quantile or probability matching \citep{gumbel:1958}, the probability weighted moment method \citep{hosking+w+w:1985}, and the maximum likelihood method \citep{prescott+w:1980, hosking:1985}.  
The present paper revisits the maximum likelihood estimator and its asymptotic distribution.

In general, deriving large-sample asymptotics of the maximum likelihood estimator for a distribution family with varying support is a difficult problem. The classical regularity conditions of \citet{cramer:1946} are not fulfilled, and the same is true for the weaker Lipschitz conditions stemming from empirical process theory \citep[Theorem~5.39]{vdV98}. Up to our knowledge, a general theory does not exist.

\citet{smith:1985} was the first to consider maximum likelihood estimation in a large class of non-regular parametric families on the real line. More precisely, he studied densities which can be written as
\[
  f(x; \mu, \phi) = (x-\mu)^{\alpha-1} \, g(x-\mu; \phi),
  \qquad x \in (\mu, \infty),
\]
where $\mu$ is a location parameter, $\phi$ is a parameter vector, $\alpha = \alpha(\phi)$ is a smooth function of $\phi$, and $g$ is a known function. The formulation is general enough to include location versions of the Weibull, Gamma, Beta and log-Gamma distribution. Depending on the value of $\alpha$, \citet{smith:1985} shows that the rate of convergence and the asymptotic distribution of the maximum likelihood estimator for $\mu$ is $n^{-1/2}$ and normal or faster than $n^{-1/2}$ and non-normal, respectively. 

The above class of densities, however, does not include the three-parameter General\-ized Extreme-Value distribution. Still, after a reparameterization, the formulation does include the case $-1/2 < \gamma < 0$. \citet[p.~88]{smith:1985} claims that similar arguments will work for $\gamma > 0$  too, but details are omitted, while the case $\gamma = 0$ is not mentioned.
It could be argued that the case $\gamma < 0$ treated in \cite{smith:1985} is indeed the most difficult one. However, we believe that his proof contains a gap which cannot be easily remedied. Uniformity in certain convergence statements is claimed but not proven, and we doubt that it can be done with the techniques used in the article. 
We justify our point of view at the end of Section~\ref{sec:mle}.

Our contribution consists of a general result on the asymptotic normality of the maximum likelihood estimator for parametric models whose support may depend on the parameter. We apply the result to the three-parameter GEV family. An interesting side result concerns the differentiability in quadratic mean of that family; for the Gumbel case, see \cite{marohn:1994, marohn:2000}.

As in \cite{smith:1985}, we need to show that certain limit relations hold uniformly over specific subsets of the parameter space. To do so, we use empirical process machinery borrowed from \citet{vdV98}. The approach requires exercising control on the entropy of certain function classes. We obtain such control via carefully formulated Lipschitz conditions. Checking these conditions for the three-parameter GEV family turns out to be surprisingly tedious.

The set-up in our paper is that of independent random samples drawn from a distribution within the parametric family itself. For the GEV family, \cite{Dom15} considers the more realistic setting of triangular arrays of block maxima extracted from 
independent random variables sampled from a distribution in the domain of attraction of a GEV distribution. He shows consistency of the maximum likelihood estimator, and our Proposition~\ref{prop:cons} below is not far from being a special case of his Theorem~2. In \cite{BucSeg15}, we go one step further and also establish asymptotic normality, even for block maxima extracted from time series. In that paper, however, we do not consider the full three-parameter GEV family but focus on the two-parameter Fr\'echet sub-family. The support of the latter is equal to the positive half-line for all values of the parameter vector. The complications that motivate the present paper do therefore not arise.

Asymptotic normality of the maximum likelihood estimator in general parametric families with parameter-dependent support is asserted in Section~\ref{sec:mle}. The proof is given in Appendix~\ref{app:mlelip}. We specialize the theory to the GEV family in Section~\ref{sec:GEV} and we provide an outline for possible applications to other parametric families in Section~\ref{sec:of}. The reasonings and calculations needed to work out the results for the GEV family are sufficiently complicated to fill Appendices~\ref{app:GEV} to~\ref{app:an}.

\section{Maximum likelihood estimator of a support-determining parameter}
\label{sec:mle}

Let $( P_\theta : \theta \in \Theta)$ be a family of distributions on some measurable space $(\Xc, \Ac)$, where $\Theta \subset \R^k$. Suppose that each $P_\theta$ has a density $p_\theta$ with respect to some common dominating measure $\mu$. The model $( P_\theta: \theta \in \Theta )$ is said to be \emph{differentiable in quadratic mean} at an inner point $\theta_0\in\Theta$ if there exists a measurable function $\dot{\ell}_{\theta_0} : \Xc \to \reals^k$ such that
\begin{align}
\label{eq:dqm}
  \int_\Xc 
    \Bigl\{ 
      \sqrt{p_{\theta_0+h}} - \sqrt {p_{\theta_0}} 
      - \frac{1}{2} h^T \dot \ell_{\theta_0} \sqrt{p_{\theta_0}} 
    \Bigr\}^2 \,
  d\mu = o(\norm{h}^2),
  \qquad h \to 0.
\end{align}
The function $\dot{\ell}_{\theta_0}$ is referred to as the \emph{score vector}. Its components are square-integrable with respect to $P_{\theta_0}$ and have mean zero. The covariance matrix $I_{\theta_0} = P_{\theta_0} \dot{\ell}_{\theta_0} \dot{\ell}_{\theta_0}^T$ is called
\emph{Fisher information matrix}. Differentiability in quadratic mean with invertible Fisher information matrix implies an asymptotic expansion of the log-likelihood ratio statistic implying \emph{local asymptotic normality} of the associated sequence of statistical experiments. For more on the statistical implications of this property, see for instance \citet{LeCam86} and \citet[Chapters~7--8]{vdV98}. In \cite{marohn:2000}, local asymptotic normality is exploited to construct asymptotically efficient tests of the Gumbel hypothesis.

Let $X_1, \dots, X_n$ be an i.i.d.\ sample from $P_{\theta_0}$. By definition, any global maximizer of the $\bar \R$-valued function $\theta \mapsto \Probn \log p_\theta = \sum_{i=1}^n \log p_\theta(X_i)$ is called a maximum likelihood estimator. Here, it is implicitly assumed that the set of global maximizers is non-empty. This is easily satisfied in most situations where $\Theta$ is compact. Otherwise, a compactification argument often works \citep[Chapter~5]{vdV98} or one can restrict attention to local maximizers instead.

Under ``regularity conditions'', the asymptotic distribution of $\sqrt{n} (\hat{\theta}_n - \theta_0)$ is $N_k( 0, I_{\theta_0}^{-1})$. Various sets of sufficient conditions have been proposed in the literature: Cram\'er's classical conditions require the existence of and bounds on the third-order derivatives of $\theta \mapsto \ell_\theta(x)= \log p_\theta(x)$. Theorem~5.39 in \citet{vdV98} only demands differentiability in quadratic mean plus a Lipschitz condition on $\theta \mapsto \ell_\theta(x)$, for all $x$ in the support of $P_{\theta_0}$ and all $\theta$ in a neighbourhood of $\theta_0$.

However, if the support, $\{ x : p_\theta(x) > 0 \}$, of $P_\theta$ depends on $\theta$, an approach based on smoothness of $\theta \mapsto \ell_\theta(x)$ is not possible. Indeed, for any neighbourhood of $\theta_0$, we may find $\theta$ in that neighbourhood and $x \in \Xc$ such that $p_{\theta_0}(x) > 0$ but $p_\theta(x) = 0$ and thus $\ell_\theta(x) = -\infty$. For the same reason, empirical processes indexed by $\theta$ in a neighbourhood of $\theta_0$ will have unbounded trajectories with probability tending to one. But weak convergence of such empirical processes in the space of bounded functions is a crucial ingredient in the proofs of many theorems on the asymptotics of M-estimators.

A possible way to circumvent these problems consists of replacing the criterion function $\ell_\theta$ by
\begin{align} 
\label{eq:m}
  m_\theta
  = 
  2 \log \left( \frac {p_\theta + p_{\theta_0} }{2p_{\theta_0}} \right),
\end{align}
a real-valued function on $\{ x : p_{\theta_0}(x) > 0 \}$ satisfying $m_\theta(x) \ge -2\log(2)$ for all such $x$. Using Corollary~5.53 in \cite{vdV98}, the function $m_\theta$ can be used to obtain the $O_{\Prob}(1/\sqrt{n})$ rate of convergence of the maximum likelihood estimator for the GEV parameter $\theta$, see Proposition~\ref{prop:rate} below. The method of proof does not allow to obtain the asymptotic distribution of the estimator, however.

Yet, if the rate of convergence $O_{\Prob}(1/\sqrt{n})$ has been established, then our next proposition provides alternative conditions (on $\ell_\theta$) which can be used to prove asymptotic normality of the maximum likelihood estimator. In doing so, we follow the well-known three-step strategy for deriving the asymptotic distribution of M-estimators (see, e.g., \citealp{vdG09}): (1) prove consistency, then (2) derive the rate of convergence and finally (3) derive the exact limiting distribution. In Section~\ref{sec:GEV}, the three steps are worked out for the GEV family by an application of the following proposition for step~(3).

The support of $P_\theta$ is $S_\theta=\{ x : p_\theta(x) > 0 \}$, a subset of $\Xc$ which may vary with $\theta \in \Theta$. Let $U_\eps(\theta) = \{\theta' \in \Theta: \norm{ \theta-\theta'} < \eps\}$ and put
\[
  \bar S(\eps) 
  = 
  \bar S(\theta_0,\eps) 
  = 
  \bigcap_{\theta \in U_{\eps}(\theta_0)} S_{\theta}
  =
  \{ x : \text{$p_\theta(x) > 0$ for all $\theta$ such that $\norm{\theta - \theta_0} < \eps$} \}.
\]
The set $\bar{S}(\eps)$ is the intersection of the supports of the distributions $P_\theta$ for all $\theta$ in an $\eps$-ball centered at $\theta_0$.
The following proposition claims the asymptotic normality of the maximum likelihood estimator under a Lipschitz condition on the function $\theta \mapsto \ell_{\theta}(x)$ for $x$ and $\theta$ in a certain range.

\begin{proposition}
\label{prop:mlelip}
Let $( P_\theta: \theta \in \Theta )$ be a parametric model on $(\Xc, \Ac)$, with $\Theta \subset \R^k$, and let $\theta_0$ be an inner point of $\Theta$. Let $X_1, X_2, \ldots$ be a sequence of independent and identically distributed random elements in $\Xc$ with common distribution $P_{\theta_0}$ and let $\hat{\theta}_n$ be a maximum likelihood estimator based on $X_1, \ldots, X_n$. Assume the following conditions:
\begin{enumerate}[(a)]
\item
The model is differentiable in quadratic mean at $\theta_0$ with score vector $\dot{\ell}_{\theta_0}$ and non-singular Fisher information matrix $I_{\theta_0}$.
\item
We have
\begin{align} 
\label{eq:suppeps}
  P_{\theta_0} \bigl( \Xc \setminus \bar S(\eps) \bigr) 
  = o(\eps^2), 
  \qquad \eps \downarrow 0.
\end{align} 
\item
There exist $\dot \ell \in L_2(P_{\theta_0})$ and $\eps_0>0$ such that, for every $0<\eps<\eps_0$, 
\begin{align} 
\label{eq:lip}
  \abs{ \ell_{\theta_1}(x) - \ell_{\theta_2}(x) } 
  \le 
  \dot \ell (x) \norm{ \theta_1 - \theta_2 } \qquad 
  \text{for all $x \in \bar S({2\eps})$ and $\theta_1, \theta_2 \in U_{\eps}(\theta_0)$}.
\end{align}
\item
We have $\sqrt n(\hat \theta_n-\theta_0)=O_\Prob(1)$ as $n \to \infty$.
\end{enumerate}
Then 
\begin{equation}
\label{eq:mlelip:an}
  \sqrt n ( \hat \theta_n - \theta_0) 
  = 
  I_{\theta_0}^{-1} \frac{1}{\sqrt n} \sum_{i=1}^n \dot \ell_{\theta_0}(X_i) + o_\Prob(1)
  \dto
  N_k( 0, I_{\theta_0}^{-1} ),
  \qquad n \to \infty.
\end{equation}
\end{proposition}

Condition~(b) controls the size of that part of the support of $P_{\theta_0}$ that is not contained in the support of $P_\theta$ for some $\theta$ in a neighbourhood of $\theta_0$. Condition~(c) is a Lipschitz condition on $\theta \mapsto \ell_\theta(x)$ in which the range of $x$ and $\theta$ is specified in such a way that $p_{\theta'}(x) > 0$ for all $\theta'$ in a neighbourhood of $\theta$. In view of~(b) and~(d), the part of the range of $x$ that has been left out is asymptotically negligible.

The proof of Proposition~\ref{prop:mlelip} adapts arguments from the proofs of Theorems~5.23,~5.39 and~7.2 and Lemma~19.31 in \cite{vdV98}. It is given in detail in Appendix~\ref{app:mlelip}.

\subsection*{The role of uniformity}

Theorem~3(i) in \citet{smith:1985} states the asymptotic normality of the maximum likelihood estimator for a certain class of real-valued parametric families whose supports depend in a smooth way on the parameter. As our paper may have the appearance of needlessly repeating known results, we feel obliged to explain our motives. The matter is not easy to explain, and we invite the reader to consult Smith's article while reading this remark. The pages in \cite{smith:1985} we will be needing are 71, 75--76, 80, and 82, and it will be convenient to discuss them in more or less reverse order.

The proof of Theorem~3(i), stated on page~80, comprises only three lines at the bottom of page~82. The heart of the argument is the claim that the second-order derivatives of the log-likelihood are asymptotically constant in a sequence of shrinking neighbourhoods of the true parameter. We agree with the author that if the claim on the second-order derivatives is true, then the asymptotic normality can be proven too, via a second-order Taylor expansion of the log-likelihood around the true parameter. Another good example of this technique is the proof of Theorem~5.41 in \citet{vdV98}.

As a justification of the assertion on the asymptotic constancy of the second-order derivatives of the log-likelihood, the author refers to the proof of his Theorem~1. Studying the proof of that theorem, we find the assertion in Lemma~4, items I(i), II(i), and III, on page~75. Only for item I(i), a proof is given (page~76), while the proof of the other items is said to be similar.
The proof of item I(i) in Lemma~4 relies on Assumption~7 on page~71. This assumption concerns the $L_1$-continuity of the second-order derivatives of the log-likelihood of a single observation as a function of the parameter vector. It is here that we wish to formulate an objection.

After careful reflection, we are convinced that Assumption~7 is insufficient to conclude that the \emph{supremum of the sample averages} on line~4 of page~76 converges to zero. Indeed, the \emph{expectation of the supremum} of a collection of random variables is in general larger than the \emph{supremum of the expectations} of those random variables. In fact, the difference is fundamental and lies at the heart of what makes proving uniform laws of large numbers and uniform central limit theorems so challenging. Controlling expectations of suprema is the topic of maximal inequalities. In \citet[Section~19.6]{vdV98}, such control is exercised by bracketing integrals of the function classes over which suprema are to be taken.

Bracketing integrals appear in the proof of our Proposition~\ref{prop:an}. We find bounds on such integrals by relying on the Donsker theory in Chapter~19 in \cite{vdV98}. The complexity of the function classes under consideration is tempered by means of uniform Lipschitz conditions. This explains the appearance of our Lipschitz condition~\eqref{eq:lip}.

\section{Maximum likelihood estimator for the GEV family}
\label{sec:GEV}

Let $X_1, \ldots, X_n$ be an independent random sample from the Generalized Extreme-Value distribution $G_{\theta_0}$ in \eqref{eq:GEV} with unknown parameter $\theta_0 = (\gamma_0, \mu_0, \sigma_0)$, to be estimated. Expressions for the log-density $\ell_\theta(x) = \log p_\theta(x)$ and the score vector $\dot{\ell}_\theta(x) = (\partial_\gamma \ell_\theta(x), \partial_\mu \ell_\theta(x), \partial_\sigma \ell_\theta(x))^T$ are given in Appendix~\ref{app:GEV}.

Consider the maximum likelihood estimator, $\hat{\theta}_n$, defined by maximizing the log-likelihood $\theta \mapsto n^{-1} \sum_{i=1}^n \ell_{\theta}(X_i) = \Probn \ell_\theta$. As argued by \cite{Dom15},  it is not guaranteed that the log-likelihood attains a unique, global maximum. For that reason, he rather defines any local maximizer of the log-likelihood over $\nd = \R \times \R \times (0,\infty)$ as a maximum likelihood estimator. His main result then states that, for block maxima constructed from an underlying i.i.d.\ series, one can always find a strongly consistent local maximizer, as long as $\gamma_0$ is strictly larger than $-1$. Going into his proofs, we see how that local maximizer is constructed: by restricting the parameter space to some arbitrary compact set $\cd\subset(-1,\infty) \times \R \times (0,\infty)$ which contains $\theta_0$ as an interior point, it is guaranteed that the $[-\infty, \infty)$-valued, continuous function $\theta \mapsto \Probn \ell_\theta$ attains its global maximum over $\cd$. That global maximum is then shown to be a local maximum over $\nd=\R \times \R \times (0,\infty)$, eventually, and to be strongly consistent.

In the subsequent parts of this paper, we could also work along Dombry's lines, re-defining maximum likelihood estimators as local rather than global maxima. However, we prefer to keep track of the above-mentioned construction within his and our proofs.  For the following proposition concerning consistency, we therefore consider the restriction of the parameter space to an arbitrarily large compact set $\cd \subset (-1,\infty) \times \R \times (0,\infty)$ that contains the true value $\theta_0\in (-1,\infty) \times \R \times (0,\infty)$ in its interior. In practice, maximizers produced by numerical algorithms are local maximizers anyway, so that the restriction of the parameter space to an arbitrarily large compact set is hardly a restriction, at least if $\gamma_0$ is known to be larger than $-1$; see \citet[Remark~4]{Dom15} for the case $\gamma_0 \le -1$.

\begin{proposition}[Consistency]
\label{prop:cons}
Let $X_1, X_2, \ldots$ be an independent random sample from the $G_{\theta_0}$ distribution, with $\theta_0 = (\gamma_0, \mu_0, \sigma_0) \in (-1, \infty) \times \reals \times (0, \infty)$.
For any compact set $\cd \subset (-1, \infty) \times \reals \times (0, \infty)$ such that $\theta_0$ is in the interior of $\cd$ and for any estimator sequence $\hat{\theta}_n$ such that $\Probn \ell_{\hat{\theta}_n} = \max_{\theta \in \cd} \Probn \ell_\theta$, such maximizers always existing, we have $\hat{\theta}_n \to \theta_0$ almost surely as $n \to \infty$.
\end{proposition}

\begin{proof}[Proof of Proposition~\ref{prop:cons}]
Since an extreme-value distribution is in its own domain of attraction, Proposition~\ref{prop:cons} is in fact a combination of Theorem~2 and Proposition~2 in \cite{Dom15} with block size sequence $m(n) \equiv 1$ and with $a_m = \sigma_0$ and $b_m = \mu_0$. However, such a choice for $m(n)$ is prohibited in the cited theorem because of its condition~(5), requiring that $m(n) / \log n \to \infty$ as $n \to \infty$. Therefore, we need to pass through the proof of the theorem and adapt where necessary.

First, the hypothesis $m(n) \to \infty$ is used in Lemmas~2 and~3 to apply the domain-of-attraction condition that $F^m(a_m x + b_m) \to \Go_{\gamma_0}(x)$ as $m \to \infty$, where $\Go_{\gamma} = G_{\gamma,0,1}$. But in our case, $F = G_{\gamma_0,\mu_0,\sigma_0}$, and thus $F(\sigma_0 x + \mu_0) = \Go_{\gamma_0}(x)$ already, without having to pass through the limit.

Further, as explained in Remark~3 in \cite{Dom15}, the condition $m(n) / \log n \to \infty$ appears only in the proof of Lemma~4 in the same paper. Writing $\ello_{\gamma} = \ell_{\gamma,0,1}$ and $Z_i = (X_i - \mu_0)/\sigma_0$, the claim of that lemma is that
\[
  \lim_{n \to \infty}
  \frac{1}{n} \sum_{i=1}^n 
  \ello_{\gamma_0}( Z_i )
  =
  \int_{\reals} \ello_{\gamma_0} ( z ) \, \diff \Go_{\gamma_0}(z),
  \qquad \text{almost surely.}
\]
But the random variables $Z_1, Z_2, \ldots$ are independent and identically distributed with common distribution $\Go_{\gamma_0}$, so that the claim follows from the strong law of large numbers. Note that the random variables $E_i = (1+\gamma Z_i)^{-1/\gamma}$ have a unit Exponential distribution and that $\ello_{\gamma_0}(Z_i) = (1+\gamma) \log(E_i) - E_i$, the absolute value of which has a finite expectation.
\end{proof}

To prove the asymptotic normality of the maximum likelihood estimator, we apply  Proposition~\ref{prop:mlelip}. To that end, we need to check a number of conditions, one of which, differentiability in quadratic mean, is interesting in its own right.

\begin{proposition}[Differentiability in quadratic mean]
\label{prop:dqm}
The three-parameter GEV family $\{ G_{(\gamma, \mu, \sigma)} : (\gamma, \mu, \sigma) \in \nd= \R \times \R \times (0,\infty) \}$ is differentiable in quadratic mean at $\theta_0 = (\gamma_0, \mu_0, \sigma_0)$ if and only if $\gamma_0 > -1/2$. In that case, the score vector $\dot{\ell}_{\theta_0}(x)$ is equal to the gradient of the map $\theta \mapsto \ell_\theta(x)$ at $\theta = \theta_0$ for $x$ such that $\sigma_0 + \gamma_0 (x - \mu_0) > 0$ and equal to $0$ otherwise.
\end{proposition}

The proof of Proposition~\ref{prop:dqm} requires showing \eqref{eq:dqm}. For every real $x$ such that $\sigma_0 + \gamma_0 (x - \mu_0) \ne 0$, the order of the integrand in \eqref{eq:dqm} is $o( \norm{h}^2 )$ as $h \to \infty$ due to the differentiability of the map $\theta \mapsto \sqrt{p_\theta(x)}$ at $\theta = \theta_0$ for such $x$. 
This asymptotic relation is true pointwise in $x$, and it remains to be shown that it can be integrated over $x$.
The most delicate case occurs when $-1/2 < \gamma_0 \le -1/3$ because in that case, the first-order partial derivatives of the map $\theta \mapsto \sqrt{p_\theta}(x)$ are not continuous at the boundary situation $\sigma + \gamma(x - \mu) = 0$. A detailed proof is given in Appendix~\ref{app:dqm}.

\begin{proposition}[Asymptotic normality]
\label{prop:an}
Let $X_1, X_2, \ldots$ be independent and identically distributed random variables with common GEV law $G_{\theta_0}$, with $\theta_0 = (\gamma_0, \mu_0, \sigma_0) \in (-1/2, \infty) \times \reals \times (0, \infty)$. Then, for any compact parameter set $\cd \subset (-1/2, \infty) \times \R \times (0,\infty)$, any sequence of maximum likelihood estimators over $\cd$ is strongly consistent and asymptotically normal:
\[ 
  \sqrt{n} ( \hat{\theta}_n - \theta_0 ) 
  = I_{\theta_0}^{-1} \frac{1}{\sqrt{n}} \sum_{i=1}^n \dot{\ell}_{\theta_0}(X_i) + o_\Prob(1)
  \dto N_3(0, I_{\theta_0}^{-1}), \qquad n \to \infty.
\]
\end{proposition}

\begin{proof}
Strong consistency follows from Proposition~\ref{prop:cons}. The proof of Proposition~\ref{prop:an} therefore consists of checking the four conditions (a)--(d) of Proposition~\ref{prop:mlelip}. Property (a) is just Proposition~\ref{prop:dqm}. The rate of convergence~(d) is the topic of Proposition~\ref{prop:rate}, the proof of which consists of an application of Corollary~5.53 in \cite{vdV98} to the modified criterion function $m_\theta$ in \eqref{eq:m}; a Lipschitz property of the latter function is established in Lemma~\ref{lem:Lipschitz}. The remaining points (b) and (c) are treated in Appendix~\ref{app:an}:  the condition (b) on the support is given in Lemma~\ref{lem:suppgev}, while the Lipschitz condition (c) is given in Lemma~\ref{lem:lipgev}.
\end{proof}

\begin{remark}
If $\gamma_0 = -1/2$, then $(d/dx) p_\theta(x)$ converges to a negative constant as $x$ increases to the upper endpoint of the support, and the results in \cite{woodroofe:1972}
suggest that the maximum likelihood estimator is still asymptotically Normal, but at a rate faster than $n^{-1/2}$; see also \citet[Theorem~3(ii)]{smith:1985}. For $-1 < \gamma_0 < -1/2$, the results in \citet{woodroofe:1974} suggest that the maximum likelihood estimator converges weakly to a certain non-Gaussian limit and at a rate depending on $\gamma_0$; see also \citet[Theorem~3(iii)]{smith:1985}. Since at $\gamma_0 \le -1/2$, the GEV family is not differentiable in quadratic mean, our Proposition~\ref{prop:mlelip} does not apply, and we do not pursue the matter further.
\end{remark}

\section{Maximum likelihood estimators for other parametric families} \label{sec:of}

The result of Proposition~\ref{prop:mlelip} is general enough to be applicable for a wide variety of parametric families with support depending on the parameter. Deriving all the details may be quite cumbersome and is beyond the scope of the present paper but, nevertheless, we would like to review a couple of parametric families on the real line where Proposition~\ref{prop:mlelip} may come in handy. 

The Generalized Pareto distribution with parameter $\theta=(\gamma,\sigma) \in \R \times (0,\infty)$ is defined by the cdf
\[
G_{\theta}(x) = 
  \begin{cases}
    1-e^{-x/\sigma}
      & \text{if $\gamma = 0$ and $x>0$}, \\
    1 -(1+\gamma x/\sigma)^{-1/\gamma}
      & \text{if $\gamma \ne 0$ and $x\in S_\theta$}, \\
     0 & \text{if  $x\le0$}, \\
          1 & \text{if  $\gamma < 0$ and $x\ge -\sigma/\gamma$,}
  \end{cases}
\]
where $S_\theta=(0,\infty)$ for $\gamma\ge0$ and $S_\theta=(0,-\sigma/\gamma)$ for $\gamma<0$. Using the notation in \eqref{eq:ugammaz} below, its density can be written as
\[
p_\theta(x) = \sigma^{-1} u_\gamma(x/\sigma)^{1+\gamma} \ind(x > 0, \, \sigma + \gamma x > 0),  \qquad x \in \R,
\]
a function which is similar to but simpler than the one of the GEV family. Along similar lines as for that family, the results from Appendix~\ref{app:GEV} can  be used to derive differentiability in quadratic mean at any point $(\gamma_0,\sigma_0)\in(-1/2, \infty) \times (0,\infty)$ (see also \citealp{marohn:1994}, for the case $\gamma_0=0$). Also, the Lipschitz condition on $\ell_\theta$ (Condition~(c) of Proposition~\ref{prop:mlelip}) and on $m_\theta$ (needed to prove Condition~(d) of Proposition~\ref{prop:mlelip} via an application of Corollary 5.53 in \citealp{vdV98}; see also the proof of Proposition~\ref{prop:rate} below) can be checked similarly to how we proceeded for the GEV family. Asymptotic normality of the MLE based on high-threshold exceedances was proved with different techniques in \cite{drees+f+dh:2004}. 

Following \cite{smith:1985}, we may also consider parametric families of densities on the real line of the form
\[
  p_{\theta}(x) = (x-\mu)^{\alpha(\phi)-1} \, g(x-\mu; \phi),
  \qquad x \in (\mu, \infty),
\]
where $\theta=(\mu,\phi)\in \Theta= \R \times \Phi \subset \R^{1+q}$ are parameters and where $g:(0,\infty) \times \Phi \to [0,\infty)$ and  $\alpha:\Phi \to (0,\infty)$ are known, smooth functions.  Among other families (see \citealp{smith:1985} for the three-parameter gamma, beta and log-gamma distributions), we obtain for instance the three-parameter Weibull family by choosing
\[
\Phi=(0,\infty)^2, \quad 
g(x;\phi_1, \phi_2)=\phi_1 \phi_2\exp(-\phi_2x^{\phi_1}), \quad
\alpha(\phi_1, \phi_2) = \phi_1.
\]
Under regularity conditions on $\alpha$ and $g$, Proposition~\ref{prop:mlelip} can be applied on the restricted parameter set $\Theta_2=\R \times \Phi_2 = \R \times \{ \phi: \alpha(\phi)>2\}$; this corresponds to case (i) of Theorem~3 of \cite{smith:1985}.

Consider Condition~(b) of Proposition~\ref{prop:mlelip}. In \cite{smith:1985}, it is assumed that 
\[
c(\phi) = \lim_{x \downarrow 0} g(x; \phi) / \alpha(\phi) \in (0,\infty)
\] 
exists (note that $c(\phi)=\phi_2>0$ for the Weibull family), which implies that $g(x; \phi_0)$ may be bounded by a  constant $K_0$ for sufficiently small $x$. As a consequence, Condition~(b) in Proposition~\ref{prop:mlelip} is met: clearly, $S_\theta=(\mu,\infty)$, so that $\bar S(\eps)=(\mu_0+\eps, \infty)$, and thus, since $\alpha_0=\alpha(\phi_0)>2$,
\begin{align*}
P_{\theta_0} \big( \R \setminus \bar S(\eps) \big) 
&= 
\int_{\mu_0}^{\mu_0+\eps} (x-\mu_0)^{\alpha_0-1} g(x-\mu_0; \phi_0) \, \diff x \\
&\le
K_0 \int_0^\eps y^{\alpha_0-1} \, \diff y = O(\eps^{\alpha_0}) = o(\eps^2), \qquad \eps \downarrow 0.
\end{align*}

Let us sketch how to arrive at the other conditions of  Proposition~\ref{prop:mlelip}. Since
\[
\ell_\theta(x) = \{ \alpha(\phi)-1\} \log(x-\mu) + \log g(x-\mu; \phi),
\]
Condition~(c) of Proposition~\ref{prop:mlelip} can  obviously be deduced from smoothness properties of the map $\phi \mapsto \alpha(\phi)$ on $\Phi_2$ and of the map $(x,\phi)\mapsto h(x,\phi)=\log g(x;\phi)$ on $(0,\infty) \times \Phi_2$. Note that 
\[
|\partial_\mu \ell_\theta(x)|  \le \left| \frac{\alpha(\phi) - 1}{x-\mu} \right| + | \partial_x h(x-\mu; \phi)|
\]
and 
\[
|\partial_\phi \ell_\theta(x)|  \le | \partial_\phi \alpha(\phi) \log(x-\mu)|  + | \partial_\phi h(x-\mu; \phi)|,
\]
showing that the inequality $\alpha(\phi)>2$ must be used again to control the integrals of the squared score functions near the lower endpoint. 
Condition~(a) of Proposition~\ref{prop:mlelip} can be conveniently shown by using Lemma~7.6 in \cite{vdV98}: if the Fisher information $I_\theta=P_\theta \dot \ell_\theta \dot \ell_\theta^T$ exists and is continuous in $\theta$, we only need to show that, for each $x\in\R$, the function $\theta\mapsto s_\theta(x) = \sqrt{p_\theta(x)}$ on $\Theta_2$ is continuously differentiable. Using the relations $\partial_\mu s_\theta(x)=\partial_\phi s_\theta(x)=0$ for $x<\mu$ and the identities
\[
\partial_\mu s_\theta(x) = \frac{1}{2} \, s_\theta(x) \, \partial_\mu \ell_{\theta}(x), \qquad 
\partial_\phi s_\theta(x) = \frac{1}{2} \, s_\theta(x) \, \partial_\phi \ell_{\theta}(x),
\]
for $x>\mu$, this follows from continuous differentiability of $\alpha$ and $h$, provided that $\alpha(\phi)>3$. For $2 < \alpha(\phi) \le 3$, $\partial_\mu s_\theta(x)$ will not converge to 0 as $x \downarrow \mu$; this is similar to the case $-1/2 < \gamma \le -1/3$ for the GEV family and may require more sophisticated arguments (see the proof of Proposition~\ref{prop:dqm}). 
Finally, Condition~(d) of Proposition~\ref{prop:mlelip} may be deduced from Corollary 5.53 in \cite{vdV98}. By following the arguments that lead to \eqref{eq:bound} in Lemma~\ref{lem:Lipschitz} (with $a=1/2$), the necessary Lipschitz condition on $m_\theta$ can again be conveniently reformulated in terms of $\alpha$ and $h$.

\appendix

\section{Proof of Proposition~\ref{prop:mlelip}}
\label{app:mlelip}

For $0<\eps<\eps_0$ and  $\theta \in U_{\eps}(\theta_0)$, define a real-valued function $\ell_{\theta, \eps}$ on $S_{\theta_0}$ by 
\[
  \ell_{\theta, \eps} (x) = \ell_\theta(x) \, \ind \{  x \in \bar S(2\eps) \}.
\]
The Lipschitz condition \eqref{eq:lip} on $\ell$ implies that
\begin{align} \label{eq:lipm}
  \abs{ \ell_{\theta_1, \eps}(x) - \ell_{\theta_2, \eps}(x) } 
  \le
  \dot{\ell}(x) \, \norm{\theta_1 - \theta_2}
\end{align}
for all $\theta_1$ and $\theta_2$ in $U_{\eps}(\theta_0)$ and for all $x\in S_{\theta_0}$.

We claim that, for any $C >0$, we have
\begin{align}
\label{eq:max}
  \Probn  \ell_{\hat \theta_n}
= \sup_{ \theta \in \Theta} \Probn \ell_\theta
\ge \sup_{ \norm{ \theta- \theta_0 } < C / \sqrt{n} } \Probn \ell_\theta 
  \ge \sup_{ \norm{ \theta- \theta_0 } < C / \sqrt{n} } \Probn \ell_{\theta, C / \sqrt{n}} - o_\Prob(1/n).
\end{align}
Only the last inequality is non-trivial. Write, for an arbitrary  $\theta$ with $\norm{ \theta- \theta_0 } < C / \sqrt{n}$,
\begin{align*}
\Probn \ell_\theta 
&= 
\Probn \ell_{\theta,C / \sqrt{n}} + \Probn \ell_\theta \ind \{ \point \in S_{\theta_0} \setminus \bar S(2C/\sqrt n) \}  \\
&\ge
\Probn \ell_{\theta,C / \sqrt{n}}
- \sup_{\norm{\theta' - \theta_0} < C / \sqrt{n} } \abs{ \Probn \ell_{\theta'} \ind \{ \point \in S_{\theta_0} \setminus \bar S(2C/\sqrt n) \}  }.
\end{align*}
The supremum on the second line is $o_\Prob(r_n)$ for any sequence $r_n\downarrow 0$: for all $\eta > 0$,
\begin{multline}
\label{eq:Szero}
  \Prob	\left[ 
    r_n^{-1} \sup_{\norm{\theta' - \theta_0} < C / \sqrt{n} }
      \abs{ \Probn \ell_{\theta'} \ind \{ \point \in S_{\theta_0} \setminus \bar S(2C/\sqrt n) \} } 
    > \eta 
  \right] 
  \\
  \le 
  \Prob\{ \exists\, i = 1, \ldots, n : X_i \in  \Xc \setminus \bar S(2C/\sqrt n) \} \\
  \le 
  n P_{\theta_0} \bigl( \Xc \setminus \bar S(2C/\sqrt n) \bigr)
  = o(1), \qquad n \to \infty,
\end{multline}
by Assumption~\eqref{eq:suppeps}. Equation~\eqref{eq:max} has thus been proved.

Now, let us show that, for fixed $C>0$ and all converging sequences $h_n\to h \in \R^k$ with $\norm{ h_n } < C$ for all $n$, we have
\begin{equation}
\label{eq:mtheta}
  \sqrt n (\ell_{\theta_0+h_n/\sqrt n, C/\sqrt n} - \ell_{\theta_0, C/\sqrt n}) 
  \to h^T \dot \ell_{\theta_0} 
  \qquad \text{ in } L_2(P_{\theta_0}).
\end{equation}
For that purpose, write the left-hand side as
\begin{align*}
  A_{n1} - A_{n2}  
  = 
  \sqrt n (\ell_{\theta_0+h_n/\sqrt n} - \ell_{\theta_0}) 
  -
  \sqrt n (\ell_{\theta_0+h_n/\sqrt n} - \ell_{\theta_0}) 
  \ind\{ \point \in S_{\theta_0} \setminus \bar S(2C/\sqrt n) \}.
\end{align*}
Because of differentiability in quadratic mean, Assumption~(a), the term $A_{n1}$ converges to $h^T \dot \ell_{\theta_0}$ in $P_{\theta_0}$-probability; see the proof of Theorem~5.39 in \cite{vdV98}. Moreover, $A_{n2}$ converges to zero in $P_{\theta_0}$-probability: for all $\eta > 0$, we have
\[
  P_{\theta_0}( \abs{ A_{n2} } > \eta ) 
  \le P_{\theta_0} \{ S_{\theta_0} \setminus \bar S(2C/\sqrt n) \} 
  = o(1), 
  \qquad n \to \infty.
\]
Hence, the convergence in  \eqref{eq:mtheta} holds in $P_{\theta_0}$-probability. In view of the Lipschitz-property of $\ell_{\theta, C/\sqrt{n}}$ in~\eqref{eq:lipm} and the fact that $P_{\theta_0} \dot{\ell}^2 < \infty$ by Assumption~(c), convergence in $P_{\theta_0}$-probability can be strengthened to $L_2(P_{\theta_0})$ con\-ver\-gence by applying the dominated convergence theorem. Indeed, for any subsequence of the right hand-side of \eqref{eq:mtheta}, we may choose a sub-subsequence along which the convergence holds almost surely. The dominated convergence theorem implies convergence in $L_2(P_{\theta_0})$ along that sub-subsequence. The claim follows since the subsequence we have started with was arbitrary.

Next, we shall show that, for any converging sequence $h_n\to h$ with $\norm{ h_n } < C$ for~all~$n$,
\begin{equation} 
  \label{eq:2ndd}
  n \, P_{\theta_0} \bigl( \ell_{\theta_0+h_n/\sqrt n, C/\sqrt n} -  \ell_{\theta_0,C/\sqrt n} \bigr) 
  \to -\frac{1}{2} h^T I_{\theta_0} h, 
  \qquad n \to \infty.
\end{equation}
Recall the empirical process $\Gb_n = \sqrt{n} ( \Probn - P_{\theta_0} )$. In view of \eqref{eq:mtheta}, computing means and variances, we have
\begin{align}
  \label{eq:bn}
  B_n(h_n) 
  = 
  \Gb_n\{ \sqrt n ( \ell_{\theta_0 + h_n/\sqrt n, C/\sqrt n} - \ell_{\theta_0, C/\sqrt n} ) - h^T \dot \ell_{\theta_0} \} 
  = o_\Prob(1),
  \qquad n \to \infty.
\end{align}
We can rewrite $B_n$ as 
\begin{align*}
B_n(h_n) 
&= 
n \Probn  ( \ell_{\theta_0 + h_n/\sqrt n, C/\sqrt n} - \ell_{\theta_0, C/\sqrt n} )  - \Gb_n h^T \dot \ell_{\theta_0} \\ 
&\hspace{1.5cm} 
- nP_{\theta_0} ( \ell_{\theta_0 + h_n/\sqrt n, C/\sqrt n} - \ell_{\theta_0, C/\sqrt n} ) \\
&=
n \Probn  ( \ell_{\theta_0 + h_n/\sqrt n} - \ell_{\theta_0} )  - \Gb_n h^T \dot \ell_{\theta_0} \\ 
&\hspace{1.5cm} 
- n\Probn ( \ell_{\theta_0 + h_n/\sqrt n} - \ell_{\theta_0} ) \ind \{ \point \in \Xc \setminus \bar S(2C/\sqrt n) \} \\
&\hspace{3cm} 
- nP_{\theta_0} ( \ell_{\theta_0 + h_n/\sqrt n, C/\sqrt n} - \ell_{\theta_0, C/\sqrt n} ).
\end{align*}
Note that $n \Probn( \ell_{\theta_0 + h_n/\sqrt{n}} - \ell_{\theta_0} ) = \sum_{i=1}^n \log (p_{\theta_0 + h_n/\sqrt{n}}/p_{\theta_0})(X_i)$ is a likelihood ratio statistic. Differentiability in quadratic mean, Assumption (a), implies
\[
  n\Probn  ( \ell_{\theta_0 + h_n/\sqrt n} - \ell_{\theta_0} )  - \Gb_n h^T \dot \ell_{\theta_0}  
  = -\frac{1}{2} h^T I_{\theta_0} h + o_\Prob(1), 
  \qquad n \to \infty;
\]
see \citet[Theorem~7.2]{vdV98}. Moreover, by Assumption~(b),
\begin{multline*}
  \Prob \left[ 
    \abs{ 
      n \, \Probn ( \ell_{\theta_0 + h_n/\sqrt n} - \ell_{\theta_0} ) \, 
      \ind \{ \point \in \Xc \setminus \bar S(2C/\sqrt n) \} 
    } 
    \ge \eta
  \right] \\
  \le 
  n \, P_{\theta_0} \bigl( \Xc \setminus \bar S(2C/\sqrt n) \bigr)
  = o(1), \qquad n \to \infty.
\end{multline*}
The last four displays imply \eqref{eq:2ndd}.

We can now follow the lines of the proof of Theorem~5.23 in \cite{vdV98} to prove the following reinforcement of \eqref{eq:bn}: for any random sequence $\tilde h_n$ such that $\norm{ \tilde h_n } < C$ almost surely for all $n$, we have
\begin{align}
\label{eq:bnt}
  B_n(\tilde h_n) = o_\Prob(1), \qquad n \to \infty.
\end{align}
To see this, note that it follows from  \eqref{eq:bn} that $B_n(h) = o_\Prob(1)$ for any fixed $h \in \reals^k$ with $\norm{h} < C$. As a consequence, the finite-dimensional distributions of the stochastic process $B_n$ converge indeed to $0$ in probability. It remains to show asymptotic tightness of the sequence $B_n$ in the space $\ell^\infty( \{ h : \norm{h} < C \} )$ equipped with the supremum distance; note that by the Lipschitz property~\eqref{eq:lipm}, the trajectories of $B_n$ are indeed bounded almost surely. Obviously, the sequence of linear processes $h \mapsto \Gb_n h^T \dot{\ell}_{\theta_0} = h^T \, \Gb_n \dot{\ell}_{\theta_0}$ is asymptotically tight, so that it suffices to show the same property for the processes 
\[
  h \mapsto M_n(h) 
  = 
  \sqrt{n} \, \Gb_n ( \ell_{\theta_0 + h/\sqrt n, C / \sqrt{n}} - \ell_{\theta_0, C / \sqrt{n}} ).
\]
This can be done along the lines of the proof of Lemma~19.31 in \cite{vdV98}, relying on a result for empirical processes indexed by sequences of function classes. More precisely, define a sequence of function classes on $S_{\theta_0}$ through
\[
  \Mcc_n = \{  \sqrt n( \ell_{\theta_0 + h/\sqrt n, C / \sqrt{n}} - \ell_{\theta_0, C / \sqrt{n}} ) : \norm{ h } <  C \}.
\]
By the Lipschitz property \eqref{eq:lipm}, the functions in $\Mcc_n$ are bounded by the envelope function $C \, \dot{\ell} \in L_2(P_{\theta_0})$. From Example~19.7 in \cite{vdV98}, we obtain the following bound on the bracketing number $N_{[\,]} \bigl( \eps , \Mcc_n, L_2(P_{\theta_0}))$: for some constant $K > 0$ not depending on $n$, we have, for all sufficiently small $\eps > 0$,
\[
  N_{[\,]} \bigl( \eps , \Mcc_n, L_2(P_{\theta_0}) \bigr) 
  \le 
  K \left( \eps^{-1} (2C^2) \, \norm{ \dot{\ell} }_{L_2(P_{\theta_0})} \right)^k,
\]
where $k$ denotes the dimension of the Euclidean space of which $\Theta$ is a subset; here we used the property that the diameter of the ball $\{ h \in \reals^k : \norm{h} < C \}$ is equal to $2C$. As a consequence, if $\delta_n \downarrow 0$, the bracketing integrals converge to zero:
\[
  J_{[\,]} \bigl( \delta_n, \Mcc_n, L_2(P_{\theta_0}) \bigr) 
  = 
  \int_0^{\delta_n} \sqrt {\log N_{[\,]} \bigl( \eps , \Mcc_n, L_2(P_{\theta_0}) \bigr)} \, d\eps \to 0,
  \qquad n \to \infty.
\]
The Lindeberg condition $P_{\theta_0} (C \dot\ell)^2 \ind \{ C \, \dot{\ell} > \eps \sqrt{n} \} \to 0$ ($n \to \infty$) is satisfied because the envelope function $C \, \dot \ell$ belongs to $L_2(P_{\theta_0})$. Asymptotic tightness of $B_n$ then follows from Theorem~19.28 in \cite{vdV98}. Equation~\eqref{eq:bnt} is thus proven.

To complete the proof of Proposition~\ref{prop:mlelip}, we can now proceed similarly to the proof of Theorem~5.23 in \cite{vdV98}, with some additional effort needed to get rid of the constant $C$. In view of \eqref{eq:2ndd}, the convergence property \eqref{eq:bnt} can be rewritten as
\begin{equation}
\label{eq:bnt:bis}
  n \, \Probn \bigl( \ell_{\theta_0 + \tilde h_n/\sqrt n, C/\sqrt{n}} - \ell_{\theta_0, C/\sqrt{n}} \bigr) 
  = 
  - \frac{1}{2} \tilde h_n^T I_{\theta_0}  \tilde h_n 
  +
  \tilde h_n^T \, \Gb_n \dot \ell_{\theta_0} 
  + 
  o_\Prob(1), \qquad n \to \infty,
\end{equation}
where $\tilde h_n$ denotes an arbitrary random sequence in $\reals^k$ with $\norm{ \tilde h_n } < C$.

Define $\hat h_{n1} = \sqrt n(\hat \theta_n - \theta_0)$ and $\hat h_{n2} = I_{\theta_0}^{-1} \Gb_n \dot \ell_{\theta_0}$. For $C>0$, let $A_{n,C}$ denote the event $\{ \max( \norm{ \hat h_{n1} }, \norm{ \hat h_{n2} }) < C \}$, and set $\tilde h_{nj} = \hat h_{nj} \ind(A_{n,C})$, for $j \in \{1,2\}$. Inserting both tilde-expressions into \eqref{eq:bnt:bis}, we get, as $n \to \infty$,
\begin{align*}
  n \, \Probn (\ell_{\theta_0+\tilde h_{n1}/\sqrt n, C/\sqrt{n}} - \ell_{\theta_0, C/\sqrt{n}}) 
  &= 
  - \frac{1}{2} \tilde h_{n1}^T I_{\theta_0}  \tilde h_{n1} 
  + \tilde h_{n1}^T\Gb_n \dot \ell_{\theta_0} 
  + o_\Prob(1), \\
  n \, \Probn (\ell_{\theta_0+\tilde h_{n2}/\sqrt n, C/\sqrt{n}} - \ell_{\theta_0, C/\sqrt{n}}) 
  &= 
  \frac{1}{2} \, \Gb_n \dot \ell^T_{\theta_0} \, I_{\theta_0}^{-1} \, \Gb_n \dot \ell_{\theta_0} \, \ind( A_{n,C} ) 
  + o_\Prob(1).
\end{align*}
Subtracting the second equation from the first one yields
\begin{align*}
  &n \, \Probn (\ell_{\theta_0+\tilde h_{n1}/\sqrt n, C/\sqrt{n}} - \ell_{\theta_0+\tilde h_{n2}/\sqrt n, C/\sqrt{n}} )  
  = 
  - Q_n \ind(A_{n,C}) + o_\Prob(1), \qquad n \to \infty, \\
  &\text{where} \quad
  Q_n
  =
    \frac{1}{2} 
    \bigl( \hat h_{n1} -  I_{\theta_0}^{-1} \, \Gb_n \dot \ell_{\theta_0} \bigr)^T \,
    I_{\theta_0} \,
    \bigl( \hat h_{n1} -  I_{\theta_0}^{-1} \, \Gb_n \dot \ell_{\theta_0} \bigr).
\end{align*}

We will show that
\begin{equation}
\label{eq:quadratic_form_to_0}
  Q_n = o_\Prob(1), \qquad n \to \infty.
\end{equation}
Since the Fisher information matrix $I_{\theta_0}$ is positive definite, this will imply that
\[
  \sqrt{n}( \hat{\theta}_n - \theta_0) - I_{\theta_0}^{-1} \frac{1}{\sqrt{n}} \sum_{i=1}^n \dot{\ell}_{\theta_0}(X_i)
  =
  \hat{h}_{n1} - I_{\theta_0}^{-1} \Gb_n \dot{\ell}_{\theta_0} 
  = o_\Prob(1), \qquad n \to \infty,
\]
which is the first part of \eqref{eq:mlelip:an}. The asymptotic normality then follows from the multivariate central limit theorem.

It remains to show \eqref{eq:quadratic_form_to_0}. Note that $Q_n \ge 0$, since $I_{\theta_0}$ is positive definite. By the maximization property \eqref{eq:max}, we have, as $n \to \infty$,
\begin{align*}
  R_n
  =
  n \, \Probn \ell_{\hat \theta_n} 
  &\ge 
  n \, \Probn \ell_{\theta_0+\tilde h_{n2}/\sqrt n, C/\sqrt n} - o_\Prob(1) \\
  &=
  n \, \Probn \ell_{\theta_0+\tilde h_{n1}/\sqrt n, C/\sqrt n} + Q_n \ind(A_{n,C}) - o_\Prob(1).
\end{align*}
Isolating $Q_n \ind(A_{n,C})$, we find
\begin{align*}
  0 \le Q_n 
  &=
  Q_n \ind(A_{n,C}) + Q_n \ind(A_{n,C}^c) \\
  &\le 
  R_n - n \, \Probn \ell_{\theta_0+\tilde h_{n1}/\sqrt n, C/\sqrt n} + Q_n \ind(A_{n,C}^c) + o_\Prob(1),
  \qquad n \to \infty.
\end{align*}
Note that $\ell_{\hat \theta_n}(x) \, \ind \{ x\in \bar S(2C/\sqrt n) \} \, \ind(A_{n,C}) = \ell_{\theta_0+\tilde h_{n1}/\sqrt n, C/\sqrt{n}}(x) \, \ind(A_{n,C})$. Therefore,
\begin{align*}
  R_n 
  &= 
  R_n \ind( A_{n,C} ) + R_n \ind( A_{n,C}^c ) \\
  &=
    \bigl[ 
	n \, \Probn  \ell_{\hat \theta_n} \ind \{  \point \in \bar S(2C/\sqrt n) \} 
	+ 
	n \, \Probn  \ell_{\hat \theta_n} \ind \{  \point \notin \bar S(2C/\sqrt n) \} 
    \bigr] \, \ind( A_{n,C} ) 
    + 
    R_n \ind(A_{n,C}^c) \\
  &= 
  n \, \Probn \ell_{\theta_0+ \tilde h_{n1} /\sqrt n, C/\sqrt n} \ind( A_{n,C}) 
  + 
  o_\Prob(1) 
  + 
  R_n \ind(A_{n,C}^c), \qquad n \to \infty;
\end{align*}
the $o_\Prob(1)$ term appears because of the argument in \eqref{eq:Szero}. Substituting the expansion for $R_n$ into the upper bound for $Q_n$ yields, as $n \to \infty$,
\begin{align*}
  0 \le Q_n 
  &\le
  n \, \Probn \ell_{\theta_0+ \tilde h_{n1} /\sqrt n, C/\sqrt n} \ind( A_{n,C}) 
  + 
  R_n \ind(A_{n,C}^c) 
  \\
  &\qquad \mbox{}
  - n \, \Probn \ell_{\theta_0+\tilde h_{n1}/\sqrt n, C/\sqrt n} + Q_n \ind(A_{n,C}^c) + o_\Prob(1) \\
  &=
  \{ R_n + Q_n - n \, \Probn \ell_{\theta_0 + \tilde h_{n1}/\sqrt n, C/\sqrt{n} }  \} \ind(A_{n,C}^c)
  + o_\Prob(1).
\end{align*}
Hence, for any $\eps>0$,
\[
  \limsup_{n\to\infty}\Prob( \abs{Q_n} > \eps  )
  \le 
  \limsup_{n\to\infty} \Prob \left( A_{n,C}^c  \right),
\]
which can be made arbitrary small by increasing $C$, using Assumption~(d). This finishes the proof of \eqref{eq:quadratic_form_to_0} and thus of Proposition~\ref{prop:mlelip}.
\qed

\section{Density and score functions of the GEV family}
\label{app:GEV}

The density, log-density, score functions and Fisher information matrix of the GEV family can of course be found in many articles and textbooks. Here we present some of these objects in a form which is convenient for later analysis. In the notation of Section~\ref{sec:mle}, the state space $(\mathcal{X}, \mathcal{A}, \mu)$ is the real line equipped with its Borel sets and the Lebesgue measure. The probability density function of the GEV distribution $P_\theta$ with parameter $\theta = (\gamma, \mu, \sigma) \in \R \times \R \times (0, \infty)$ is given by
\[
  p_\theta(x) 
  = \sigma^{-1} \, e^{-u} u^{\gamma + 1} \, \ind( 1 + \gamma z > 0 ), 
  \qquad x \in \R,
\]
where
\begin{align}
\nonumber
  z &= z_{\mu,\sigma}(x)= \frac{x - \mu}{\sigma}, \\
\label{eq:ugammaz}
  u &=u_\gamma(z) = \exp \left( - \int_0^z \frac{1}{1+\gamma t} \, dt \right) =
  \begin{cases} 
    (1 + \gamma z)^{-1/\gamma} & \text{if $\gamma \ne 0$,} \\ 
    e^{-z} & \text{if $\gamma = 0$.} 
  \end{cases}
\end{align}
The expression $u_\gamma(z)$ is convex and decreasing in $z$ and is increasing in $\gamma$. See Figure~\ref{fig:ugammaz} for graphs of the functions $z \mapsto u_\gamma(z)$ for $\gamma \in \{-0.5, 0, 0.5\}$.

\begin{figure}
\begin{center}
\includegraphics[width=0.4\textwidth]{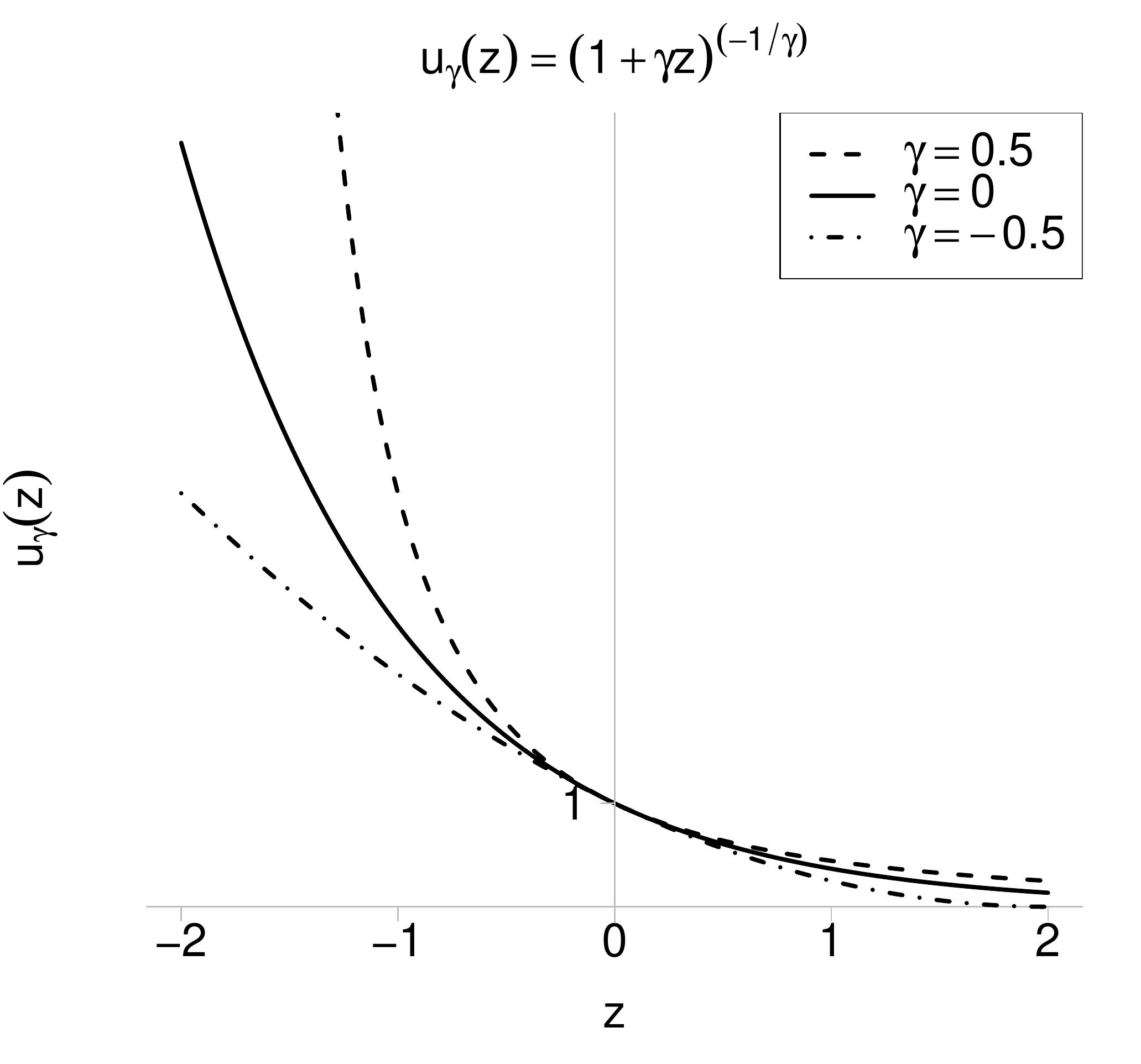}
\end{center}
\caption{\label{fig:ugammaz} Graph of $z \mapsto u_\gamma(z)$ in \eqref{eq:ugammaz} for $\gamma \in \{-0.5, 0, 0.5\}$ and $z \in [-2, 2]$.}
\end{figure}

The support, $S_\theta$, of the distribution is (defined as) the open interval
\[
  S_\theta=\{ x : p_\theta(x) > 0 \} = \{ x \in \R : \sigma + \gamma(x - \mu) > 0 \} =
  \begin{cases}
    (-\infty, \mu-\sigma/\gamma) & \text{if $\gamma < 0$,} \\
    \R & \text{if $\gamma = 0$,} \\
    (\mu-\sigma/\gamma, \infty) & \text{if $\gamma > 0$.}
  \end{cases}
\]
The log-density is given by
\[
  \ell_\theta(x) = - \log(\sigma) - u + (\gamma + 1) \log(u), \qquad x \in S_\theta,
\]
with $\log(u) = - \int_0^z (1 + \gamma t)^{-1} \, dt$. The partial derivatives of the map $\theta \mapsto \ell_\theta(x)$ at $\theta$ such that $p_\theta(x) > 0$ are given by
\begin{align}
\label{eq:scoregamma}
  \partial_\gamma \ell_\theta(x) &= (1 - u) \, \partial_\gamma \log(u) - \frac{z}{1 + \gamma z}, \\
\label{eq:scoremu}
  \partial_\mu \ell_\theta(x) &= \frac{\gamma + 1 - u}{\sigma(1 + \gamma z)}, \\
\label{eq:scoresigma}
  \partial_\sigma \ell_\theta(x) &= \frac{(1-u) z - 1}{\sigma (1 + \gamma z)},
\end{align}
with $\partial_\gamma \log(u)$ given by
\begin{align}
\label{eq:dlogu}
  0 \le \partial_\gamma \log(u) 
  &= \int_0^z \frac{t}{(1 + \gamma t)^2} \, dt 
  =
  \begin{cases}
    \dfrac{1}{\gamma} \left( \dfrac{1}{\gamma} \log(1 + \gamma z) - \dfrac{z}{1+\gamma z} \right) 
      & \text{if $\gamma \ne 0$,} \\[1em]
    \dfrac{z^2}{2} 
      & \text{if $\gamma = 0$.}
  \end{cases}
\end{align}
For $\gamma > - 1/2$, these partial derivatives have expectation zero and finite second moments with respect to $P_\theta$. For such $\theta$, the Fisher information matrix $I_\theta \in \R^{3 \times 3}$ is equal to the covariance matrix of the score vector 
\[ 
  \dot{\ell}_\theta 
  = ( \partial_\gamma \ell_\theta, \partial_\mu \ell_\theta, \partial_\sigma \ell_\theta )^T,
\]
the three entries of which are viewed as elements of $L_2( P_\theta )$. Explicit expressions for $I_\theta$ are given in \cite{prescott+w:1980}, but we will not be needing those here. The only properties of $I_\theta$ that are relevant for our current study are that $I_\theta$ is symmetric, positive definite and non-singular and that the map $\theta \mapsto I_\theta$ is continuous.

We continue with a number of properties of the GEV densities and score functions.

\begin{lemma}
Let $\gamma$ and $z$ be such that $1 + \gamma z > 0$. Let $u = u_\gamma(z)$ be as in \eqref{eq:ugammaz}. If $\gamma z > 0$, then
\begin{equation}
\label{eq:dlogu:pos}
 \partial_\gamma \log(u)
 \le 
 \left\{ 
   \begin{array}{l}
     \dfrac{z^2}{2}, \\[1em]
     \dfrac{1}{\gamma} \log(1 + \gamma z) \, \dfrac{z}{1 + \gamma z} \le 
     \left\{
	\begin{array}{l}
	\dfrac{z^2}{1 + \gamma z}, \\[1em]
	\dfrac{1}{\gamma^2} \log(1 + \gamma z).
	\end{array}
     \right.
   \end{array}
 \right.
\end{equation}
If $-1 < \gamma z < 0$, then also
\begin{equation}
\label{eq:dlogu:neg}
 \partial_\gamma \log(u)
 \le \frac{z^2}{1 + \gamma z}.
\end{equation}
As a consequence, for any $\gamma$ and $z$ such that $1+\gamma z>0$, we have
\begin{align} \label{eq:dlogu:bound}
0 \le \partial_\gamma \log(u)  \le \frac{z^2}{1+\gamma z}.
\end{align}
\end{lemma}

\begin{proof}
If $\gamma = 0$ or $z = 0$, the stated inequalities are clearly satisfied, so suppose that $\gamma \ne 0$ and $z \ne 0$. Substituting $\gamma t = v$ in \eqref{eq:dlogu}, we find
\[
 \partial_\gamma \log(u) = \frac{1}{\gamma^2} \int_0^{\gamma z} \frac{v}{(1+v)^2} \, dv.
\]
If $\gamma z > 0$, since $v \mapsto v/(1+v)$ is increasing in $v \ge 0$ and $v \mapsto 1/(1+v)$ is decreasing in $v \ge 0$, we obtain the bounds in the first display of the lemma: not only
\[
 \partial_\gamma \log(u)  \le  \frac{1}{\gamma^2} \int_0^{\gamma z} v \, dv = \frac{z^2}{2},
\]
but also
\[
 \partial_\gamma \log(u) 
 \le \frac{z}{1 + \gamma z} \, \frac{1}{\gamma} \int_0^{\gamma z} \frac{1}{1+v} \, dv
 \le \frac{z^2}{1+\gamma z}.
\]
If $\gamma z < 0$, then we have $\gamma z = - \abs{ \gamma z }$ and
\begin{align*}
 \partial_\gamma \log(u) 
 &= \frac{1}{\gamma^2} \int_0^{\abs{\gamma z}} \frac{v}{(1-v)^2} \, dv 
 \le \frac{\abs{z}}{\abs{\gamma}} \int_0^{\abs{\gamma z}} \frac{1}{(1-v)^2} \, dv
 = \frac{\abs{z}^2}{1 - \abs{\gamma z}}.  \qedhere
\end{align*} 
\end{proof}

\begin{lemma}
\label{lem:score:bounds:mu}
Let $\theta$ and $x$ be such that $1 + \gamma z > 0$, where $z = (x - \mu)/\sigma$. We have
\[
  \abs{ \partial_\mu \ell_\theta(x) }
  \le
  \begin{cases}
    \sigma^{-1} (1 + \abs{\gamma}) \, u^{1+\gamma} & \text{if $z \le 0$,} \\[1ex]
    \sigma^{-1} (1 + \abs{\gamma}) \, u^{\gamma} & \text{if $z \ge 0$.}
  \end{cases}
\]
\end{lemma}

\begin{proof}
We have $(1+\gamma z)^{-1} = u^\gamma$ and thus $\partial_\mu \ell_\theta(x) = \sigma^{-1} (\gamma+1-u) u^\gamma$. If $z \le 0$, then $u \ge 1$, so that $\abs{ \gamma + 1 - u } \le u + \abs{\gamma} \le (1 + \abs{\gamma}) u$. If $z \ge 0$, then $0 < u \le 1$, so that $\abs{ \gamma + 1 - u } \le 1 + \abs{\gamma}$.
\end{proof}

\begin{lemma}
\label{lem:score:bounds:sigma}
Let $\theta$ and $x$ be such that $1 + \gamma z > 0$, where $z = (x - \mu)/\sigma$. We have
\[
  \abs{ \partial_\sigma \ell_\theta(x) }
  \le
  \begin{cases}
    \dfrac{z+1}{\sigma (1 + \gamma z)} & \text{if $z \ge 0$,} \\[1em]
    \sigma^{-1} \bigl( 1 + u \log(u) \bigr) \, u^{\max(\gamma, 0)} & \text{if $z \le 0$.}
  \end{cases}
\]
\end{lemma}

\begin{proof}
If $z \ge 0$, then $0 < u \le 1$, so that $\abs{ (1-u)z - 1 } \le z + 1$, yielding the stated bound.

If $z \le 0$, then $u \ge 1$, so that $\abs{(1 - u)z - 1 } \le u \abs{z} + 1$. Further, $(1 + \gamma z)^{-1} = u^\gamma$ as well as
\begin{equation}
\label{eq:z1pgz}
  \frac{\abs{z}}{1+\gamma z}
  =
  \frac{u^\gamma - 1}{\gamma}
  =
  \int_1^u t^{\gamma-1} \, \diff t
  \le
  u^{\max(\gamma, 0)} \int_1^u t^{-1} \, \diff t
  =
  u^{\max(\gamma, 0)} \log(u).
\end{equation}
The stated bound now follows from
$
  \abs{ \partial_\sigma \ell_\theta(x) }
  \le
  \sigma^{-1}
  \frac{u \abs{z} + 1}{(1 + \gamma z)}
$.
\end{proof}

\begin{lemma}
\label{lem:score:bounds}
Let $\theta$ and $x$ be such that $1 + \gamma z > 0$, where $z = (x-\mu)/\sigma$.\\
If $\gamma \ge 0$ and $z \ge 0$, then, {with $\gamma^{-2}\log(1+\gamma z)$ and $\gamma^{-1}$ being defined as $+\infty$ for $\gamma=0$,}
\begin{align}
\label{eq:scoregamma:++}
 \abs{ \partial_\gamma \ell_\theta(x) }
 &\le
 \max \left\{ 
   \min \left( 
     \frac{z^2}{2}, \frac{1}{\gamma^2} \log(1 + \gamma z)
   \right),
   \min \left(
     z, \frac{1}{\gamma}
   \right)
 \right\}.
\end{align}
If $\gamma \le 0$ and $z \ge 0$, then
\begin{align}
\label{eq:scoregamma:-+}
 \abs{ \partial_\gamma \ell_\theta(x) }
 &\le \frac{\max(z^2, z)}{1 + \gamma z}.
\end{align}
If $\gamma \ge 0$ and $z \le 0$, then
\begin{align}
\label{eq:scoregamma:+-}
 \abs{ \partial_\gamma \ell_\theta(x) }
 &\le u^{1 + \gamma} \max \{ (\log u)^2, \log u \}.
\end{align}
If $\gamma \le 0$ and $z \le 0$,
\begin{align}
\label{eq:scoregamma:--}
 \abs{ \partial_\gamma \ell_\theta(x) }
 &\le u \, \max \{ (\log u)^2, \log u \} .
\end{align}
\end{lemma}

\begin{proof}
If $\gamma = 0$, then $u = e^{-z}$ and $\abs{\partial_\gamma \ell_{\theta}(x)} = (1-e^{-z})z^2/2 - z$, and all stated inequalities are satisfied. In the remainder of the proof, we assume therefore that $\gamma \ne 0$.

Suppose first that $z \ge 0$, so that $0 < u \le 1$. Since  $\abs{a-b} \le \max(a,b)$ for non\-negative numbers $a,b$,
we have
\[
 \abs{ \partial_\gamma \ell_\theta(x) }
 \le \max \left\{ \partial_\gamma \log(u), \frac{z}{1 + \gamma z} \right\}.
\]
Use \eqref{eq:dlogu:pos} to find \eqref{eq:scoregamma:++} and use \eqref{eq:dlogu:neg} to find \eqref{eq:scoregamma:-+}.

Next suppose $z \le 0$, so that $u \ge 1$. If $\gamma < 0$, then, by \eqref{eq:dlogu:pos} and \eqref{eq:z1pgz},
\[
 \partial_\gamma \log(u)
 \le (-\frac{1}{\gamma}) \log(1 + \gamma z) \, \frac{\abs{z}}{1 + \gamma z}
 \le (\log u)^2.
\]
We find, again using \eqref{eq:z1pgz}, as stated in \eqref{eq:scoregamma:--}, that
\begin{align*}
 \abs{ \partial_\gamma \ell_\theta(x) }
 \le \max \left\{ u \, \partial_\gamma \log(u), \frac{ \abs{z} }{1 + \gamma z} \right\} 
 \le u \, \max \{ (\log u)^2, \log u \}.
\end{align*}

Finally, suppose $z \le 0$ and $\gamma > 0$. By \eqref{eq:dlogu:neg} and \eqref{eq:z1pgz}
\[
 \partial_\gamma \log(u) \le \frac{z^2}{1+\gamma z} = (1+\gamma z) \left(\frac{\abs{z}}{1+\gamma z}\right)^2
 \le u^{-\gamma} \bigl( u^\gamma \log (u) \bigr)^2 = u^\gamma (\log u)^2.
\]
We obtain, again using \eqref{eq:z1pgz},
\[
 \abs{ \partial_\gamma \ell_\theta(x) }
 \le \max \left\{ u \, \partial_\gamma \log(u), \frac{\abs{z}}{1 + \gamma z} \right\}
 \le u^{1 + \gamma} \max \{ (\log u)^2, \log u \},
\]
which is \eqref{eq:scoregamma:+-}.
\end{proof}

\begin{lemma}
\label{lem:pdf:pow}
Fix $a\in[1/2,1)$. For any $x \in \R$, the function $\theta \mapsto p_\theta^a(x)$ is continuously differentiable on $(-a/(1+a), \infty) \times \R \times (0, \infty)$ and the partial derivatives are continuous in $(x, \theta) \in \R \times (-a/(1+a), \infty) \times \R \times (0, \infty)$.
\end{lemma}

\begin{proof}
Fix $\theta_0 = (\gamma_0, \mu_0, \sigma_0)$ with $\gamma_0 > -a/(1+a)$ and choose $x_0 \in \R$.

If $\sigma_0 + \gamma_0 ( x_0 - \mu_0 ) > 0$, then also $\sigma + \gamma (x - \mu) > 0$ for all $(x, \theta)$ in a neighbourhood of $(x_0, \theta_0)$ and thus $p_\theta^a(x) = \sigma^{-a} e^{-au} u^{a(\gamma+1)}$ for such $(x, \theta)$. All functions arising in the expression of $p_\theta^a(x)$ are continuously differentiable in the three parameters and are continuous in $(x, \theta)$. Since
\begin{equation*}
 \partial_{\theta_k} p_\theta^a(x)
  = a \, p_\theta^a(x) \,\partial_{\theta_k} \ell_\theta(x),
\end{equation*}
the formulas for the score function in \eqref{eq:scoregamma}, \eqref{eq:scoremu}, and \eqref{eq:scoresigma} imply that
\begin{align*}
  \partial_\gamma p_\theta^a(x) &= \left\{ (1 - u) \, \partial_\gamma \log(u) - \frac{z}{1 + \gamma z} \right\} \frac{a}{\sigma^a} e^{-ua} u^{a(\gamma+1)}, \\
  \partial_\mu p_\theta^a(x) &= \frac{\gamma + 1 - u}{\sigma(1 + \gamma z)} \frac{a}{\sigma^a} e^{-ua} u^{a(\gamma+1)}, \\
  \partial_\sigma p_\theta^a(x) &= \frac{(1-u) z - 1}{\sigma (1 + \gamma z)} \frac{a}{\sigma^a} e^{-ua} u^{a(\gamma+1)}.
\end{align*} 

If $\sigma_0 + \gamma_0 (x_0 - \mu_0) < 0$, then also $\sigma + \gamma (x - \mu) < 0$ for all $(x,\theta)$ in a neighbourhood of $(x_0,\theta_0)$ and thus $p_\theta^a(x) = 0$ for all such $(x, \theta)$, whence the partial derivatives vanish too.

The difficult case is $\sigma_0 + \gamma_0 ( x_0 - \mu_0) = 0$, that is, if $\gamma_0 \ne 0$ and $x_0 = \mu_0 - \sigma_0 / \gamma_0$. We need to show that, for every $k \in \{1, 2, 3\}$,
\[
  \lim_{\substack{ (x, \theta) \to (x_0, \theta_0) \\ \sigma + \gamma(x - \mu) > 0} }
  \partial_{\theta_k} p_\theta^a(x) = 0
\]
where $(\theta_1, \theta_2, \theta_3) = (\gamma, \mu, \sigma)$. Recall $u = (1 + \gamma z)^{-1/\gamma}$ with $z = (x - \mu) / \sigma$.

First, suppose that $\gamma_0 > 0$. Then $u \to \infty$ as $(x, \theta) \to (x_0, \theta_0)$ and convergence to zero is assured by the exponential factor $e^{-au}$ in each of the three partial derivatives.

Second, suppose that $\gamma_0 < 0$. Then $u \to 0$ as $(x, \theta) \to (x_0, \theta_0)$.  Using the bound in \eqref{eq:dlogu:bound}, we see that the limit behaviour of the three partial derivatives is dominated by the factor
\[
  (1 + \gamma z)^{ - \frac{a}{\gamma} - a - 1 }
\]
as $1 + \gamma z \to 0$. The exponent must be positive eventually:
$
  - (a/\gamma_0) - a - 1 > 0
$.
But this is equivalent to $-a/(1+a) < \gamma_0 < 0$.
\end{proof}

\section{Differentiability in quadratic mean of the GEV family}
\label{app:dqm}

\begin{proof}[Proof of Proposition~\ref{prop:dqm}]
Let $s_\theta = \sqrt{p_\theta}$. We distinguish between three cases: $\gamma_0 > -1/3$, $-1/2 < \gamma_0 \le -1/3$, and $\gamma_0 \le -1/2$.
\medskip

\noindent\textit{Case $\gamma_0 > -1/3$.} Lemma~\ref{lem:pdf:pow} implies that for all $x \in \reals$, the function $\theta \mapsto s_\theta(x)$ is continuously differentiable in a neighbourhood of $\theta_0$. Differentiability in quadratic mean then follows from an application of Lemma~7.6 in \cite{vdV98}. Note that the existence and the continuity of the Fisher information matrix in a neighbourhood of $\theta_0$ has been derived in \cite{prescott+w:1980}.
\medskip

\noindent\textit{Case $-1/2 < \gamma_0 \le -1/3$.} Recall $s_\theta = \sqrt{p_\theta}$. The conditions of Lemma~7.6 in \cite{vdV98} are no longer fulfilled, but an adaptation of that proof still yields differentiability in quadratic mean.

Since the unit ball in $\reals^3$ is compact, it suffices to show that, if $h_n \to h$ in $\reals^3$ and if $t_n \downarrow 0$ as $n \to \infty$, then
\[
  \lim_{n \to \infty} 
  \int_{\reals} 
    \left(
      \frac{s_{\theta_0 + t_n h_n}(x) - s_{\theta_0}(x)}{t_n} 
      - 
      \frac{1}{2} \, h_n^T \, \dot{\ell}_{\theta_0}(x) \, s_{\theta_0}(x)
    \right)^2
  \diff x
  = 0.
\]
It is sufficient to show the same equality with $h_n^T \, \dot{\ell}_{\theta_0}(x) \, s_{\theta_0}(x)$ replaced by $h^T\, \dot{\ell}_{\theta_0}(x) \, s_{\theta_0}(x)$; to see why, use the elementary inequality $(a+b)^2 \le 2a^2 + 2b^2$, the fact that the score vector has $P_{\theta_0}$-square integrable components, and the assumption that $h_n \to h$ as $n \to \infty$. For every $x \in \reals$ except $x = \mu_0 - \sigma_0 / \gamma_0$, the integrand of the resulting integral converges to zero by differentiability of the map $\theta \mapsto s_\theta(x)$ at $\theta = \theta_0$. What remains is to show convergence of the integral itself. To this end, we apply Proposition~2.29 in \cite{vdV98} with $p = 2$; the condition to check is that
\begin{equation}
\label{eq:limsup_to_show}
  \limsup_{n \to \infty} 
  \int_{\reals} 
    \left( 
      \frac{s_{\theta_0 + t_n h_n}(x) - s_{\theta_0}(x)}{t_n} 
    \right)^2 \, 
  \diff x
  \le
  \int_{\reals}
    \left(
      \frac{1}{2} \, h^T \, \dot{\ell}_{\theta_0}(x) \, s_{\theta_0}(x)
    \right)^2
  \diff x.
\end{equation}
The right-hand side in \eqref{eq:limsup_to_show} is equal to $(1/4) h^T (\int \dot{\ell}_{\theta_0} \dot{\ell}_{\theta_0}^T \, p_{\theta_0}) h = (1/4) h^T I_{\theta_0} h$. We need to control the left-hand side.

Write $h_n = (h_{n1}, h_{n2}, h_{n3})^T$ and $\theta_n = \theta_0 + t_n h_n = (\gamma_n, \mu_n, \sigma_n)$. Without loss of generality, assume that $n$ is large enough such that $-1/2 < \gamma_n < 0$. The upper endpoints of the supports of the GEV distributions with parameter vectors $\theta_0$, $\theta_n$, $(\gamma_0, \mu_n, \sigma_n)$ and $(\gamma_n, \mu_0, \sigma_0)$ are equal to $\mu_0 - \sigma_0/\gamma_0$, $\mu_n - \sigma_n/\gamma_n$, $\mu_n - \sigma_n / \gamma_0$ and $\mu_0 - \sigma_0/\gamma_n$, respectively. Let $\alpha_{n}$ and $\beta_{n}$ be the minimum and the maximum of these four endpoints, respectively. Write
\[
  f_n = \frac{s_{\theta_0 + t_n h_n} - s_{\theta_0}}{t_n}.
\]
Since $f_n(x) = 0$ for $x > \beta_n$, we have
\[
  \int_{\reals} f_n(x)^2 \, \diff x
  =
  \int_{-\infty}^{\alpha_n} f_n(x)^2 \, \diff x + \int_{\alpha_n}^{\beta_{n}} f_n(x)^2 \, \diff x.
\]
We will show that the limit superior of the first term on the right-hand side is bounded by the right-hand side of \eqref{eq:limsup_to_show}, while the second term on the right-hand side converges to zero.

To bound the integral of $f_n(x)^2$ from $-\infty$ to $\alpha_n$, we can proceed as in the proof of Lemma~7.6 in \cite{vdV98}. See in particular the display on top of page~96. Each point $x \in (-\infty, \alpha_n)$ is an element of the support of $s_\theta$ for each $\theta$ on the line segment connecting $\theta_0$ and $\theta_n$; this was the reason for introducing the additional parameter vectors $(\gamma_0, \mu_n, \sigma_n)$ and $(\gamma_n, \mu_0, \sigma_0)$ in the previous paragraph. Let $\dot{s}_\theta(x)$ denote the gradient of the map $\theta \mapsto s_{\theta}(x)$, for $x$ in the support of $P_\theta$. Note that $\dot{s}_\theta(x) = (1/2) s_\theta(x) \, \dot{\ell}_\theta(x)$. Since the map $u \mapsto s_{\theta_0 + ut_nh_n}(x)$, for $u \in [0, 1]$, is continuously differentiable, we have $f_n(x) = \int_{u=0}^1 \dot{s}_{\theta_0 + u t_n h_n}(x)^T \, h_n \, \diff u$. Applying the Cauchy--Schwarz inequality and the Fubini theorem, we find
\begin{align*}
  \int_{-\infty}^{\alpha_n} 
    f_n(x)^2 \, 
  \diff x
  &\le 
  \int_{x=-\infty}^{\alpha_n} 
    \int_{u=0}^1 
      \bigl( \dot{s}_{\theta_0 + u t_n h_n}(x)^T \, h_n \bigr)^2 \, 
    \diff u \, 
  \diff x \\
  &\le 
  \frac{1}{4} 
  \int_{u=0}^1 
    h_n^T I_{\theta_0 + u t_n h_n} h_n \, 
  \diff u.
\end{align*}
By continuity of the map $\theta \mapsto I_\theta$, the right-hand side converges to $(1/4) h^T I_{\theta_0} h$.

To show that $\int_{\alpha_n}^{\beta_n} f_n(x)^2 \, \diff x$ converges to zero, observe that
\[
  f_n(x)^2 \le 2 t_n^{-2} \bigl( p_{\theta_n}(x) + p_{\theta_0}(x) \bigr).
\]
Moreover, there exists $c_0 > 0$ such that $\alpha_n \ge \beta_n - c_0 t_n$ for all sufficiently large $n$. As a consequence, it is sufficient to show that, whenever $\tilde{\theta}_n \to \theta_0$ and $u_n \downarrow 0$, we have
\begin{equation}
\label{eq:limit_to_show}
  \lim_{n \to \infty} u_n^{-2} \, P_{\tilde{\theta}_n} [ \tilde{\omega}_n - u_n, \tilde{\omega}_n ) = 0,
\end{equation}
with $\tilde{\omega}_n$ the upper endpoint of the support of $P_{\tilde{\theta}_n}$. Writing $\tilde{\theta}_n = (\tilde{\gamma}_n, \tilde{\mu}_n, \tilde{\sigma}_n)$, we have
\begin{align*}
  P_{\tilde{\theta}_n} [ \tilde{\omega}_n - u_n, \tilde{\omega}_n )
  &= 
  1 - \exp 
  \left[ 
    - \left\{ 
      1 + \tilde{\gamma}_n 
      \frac%
	{\tilde{\mu}_n - \tilde{\sigma}_n / \tilde{\gamma}_n - u_n - \tilde{\mu}_n}%
	{\tilde{\sigma}_n}
      \right\}^{-1/\tilde{\gamma}_n}
  \right] \\
  &=
  1 - \exp 
  \left\{ 
    - \left( \abs{\tilde{\gamma}_n} u_n / \tilde{\sigma}_n \right)^{1/\abs{\tilde{\gamma}_n}} 
  \right\} \\
  &\le
  \left( \abs{\tilde{\gamma}_n} u_n / \tilde{\sigma}_n \right)^{1/\abs{\tilde{\gamma}_n}}.
\end{align*}
On the last line, we used the elementary inequality $1 - \exp(-z) \le z$ for all $z \in \reals$. Since $1/\abs{\tilde{\gamma}_n} \to 1/\abs{\gamma_0} > 2$, equation~\eqref{eq:limit_to_show} follows.

\medskip

\noindent \textit{Case $\gamma_0 \le -1/2$.} 
We will show that for $h = (0, t, 0)^T$ with $t \downarrow 0$, we have 
\begin{equation}
\label{eq:liminf_to_show}
  \liminf_{t \downarrow 0} t^{-2} P_{\theta_0 + h} \{ p_{\theta_0} = 0 \} > 0.
\end{equation}
Here, $\{ p_{\theta_0} = 0 \}$ is short-hand for $\{ x \in \reals : p_{\theta_0}(x) = 0 \} = [\mu_0 - \sigma_0 / \gamma_0, \infty )$. Restricting the integral on the left-hand side in \eqref{eq:dqm} to the set $\{ p_{\theta_0} = 0 \}$ then shows that the asymptotic relation in \eqref{eq:dqm} cannot hold.

We show \eqref{eq:liminf_to_show}. The upper endpoint of $P_{\theta_0 + h}$ is equal to $\mu_0 + t - \sigma_0/\gamma_0$, which is larger than the one of $P_{\theta_0}$. The mass assigned by $P_{\theta_0+h}$ to the difference of the two supports is given by
\begin{align*}
  P_{\theta_0 + h} [ \mu_0 - \sigma_0/\gamma_0, \infty )
  &= 1 - \exp \left[ - \left\{ 1 + \gamma_0 \frac{(\mu_0 - \sigma_0/\gamma_0) - \mu_0 - t}{\sigma_0} \right\}^{-1/\gamma_0} \right] \\
  &= 1 - \exp \left\{ - \left( \abs{\gamma_0} t / \sigma_0 \right)^{1/\abs{\gamma_0}} \right\}.
\end{align*}
Since $0 < 1 / \abs{\gamma_0} \le 2$ and since $1 - \exp(-u) = (1 + o(1)) \, u$ as $u \to 0$, the inequality \eqref{eq:liminf_to_show} follows.
\end{proof}

\section{Rate of convergence}

To apply Proposition~\ref{prop:mlelip}, the $O_\Prob(1/\sqrt{n})$ rate of convergence of the maximum likelihood estimator needs to be established first.

\begin{proposition}[Rate of convergence]
\label{prop:rate}
Let $X_1, X_2, \ldots$ be independent and identically distributed random variables with common GEV law $G_{\theta_0}$,  with $\theta_0 = (\gamma_0, \mu_0, \sigma_0) \in (-1/2, \infty) \times \reals \times (0, \infty)$. Then, for any compact parameter set $\cd \subset (-1/2, \infty) \times \R \times (0,\infty)$, any sequence of maximum likelihood estimators over $\cd$ satisfies $\hat{\theta}_n - \theta_0 = O_\Prob(1/\sqrt{n})$ as $n \to \infty$.
\end{proposition}

\begin{proof}[Proof of Proposition~\ref{prop:rate}]
We apply Corollary~5.53 in \cite{vdV98} to the function $m_\theta$ in \eqref{eq:m}. To do so, we need to check a number of things:
\begin{itemize}
\item
$\hat{\theta}_n = \theta_0 + o_\Prob(1)$ as $n \to \infty$: this is okay by Proposition~\ref{prop:cons}.
\item
$\Probn m_{\hat{\theta}_n} \ge \Probn m_{\theta_0} - O_\Prob(n^{-1})$: By concavity of the logarithm, we have
\begin{align*} 
\Probn m_{\hat \theta_n}
\ge 
\Probn\left( \log \frac{p_{\hat \theta_n}}{p_{\theta_0}}\right)  + \Probn \log 1
\ge 0
= \Probn m_{\theta_0}.
\end{align*}
\item
There exists $\dot{m} \in L_2(P_{\theta_0})$ such that $\abs{m_{\theta_1}(x) - m_{\theta_2}(x)} \le \dot{m}(x) \norm{\theta_1 - \theta_2}$ for $P_{\theta_0}$-almost all $x$ and all $\theta_1$ and $\theta_2$ in a neighbourhood of $\theta$: this Lipschitz condition is the topic of Lemma~\ref{lem:Lipschitz}.
\item
The map $\theta \mapsto P_{\theta_0} m_\theta$ admits a second-order Taylor expansion at the point of maximum $\theta_0$ with non-singular second derivative: this is established in Lemma~\ref{lem:second_order_expansion}.
\end{itemize}
The cited corollary now yields $\sqrt{n} (\hat{\theta}_n - \theta_0) = O_{\Prob}(1)$ as $n \to \infty$, as required.
\end{proof}

\begin{lemma}[Relative errors]
\label{lem:relerr}
Let $(\mu_0, \sigma_0) \in \reals \times (0, \infty)$ and $\eps \in (0, 1]$. Let $(\mu, \sigma) \in \reals \times (0, \infty)$ be such that $\abs{\mu - \mu_0}/\sigma_0 \le \eps$ and $\abs{ \sigma / \sigma_0 - 1 } \le \eps$. Let $x \in \reals$ and write $z_0 = (x - \mu_0)/\sigma_0$ and $z = (x - \mu)/\sigma$. If $\abs{z_0} \ge 2$, then $\abs{z/z_0 - 1} \le 2\eps$. 
\end{lemma}

\begin{proof}
Since $(1+b)(1-c)-1=b-c+bc$ for $b,c\in\R$, we have
\begin{align*}
\nonumber
  \abs{\frac{z}{z_0} - 1}
  &= \abs{ \left( 1 + \frac{\sigma_0}{\sigma} - 1 \right) 
  \left( 1 - \frac{(\mu-\mu_0)/\sigma_0}{(x-\mu_0)/\sigma_0} \right) - 1 } \\
  &\le \abs{\frac{\sigma_0}{\sigma} - 1 }
  + \abs{ \frac{\mu - \mu_0}{2 \sigma_0} }
  + \abs{\frac{\sigma_0}{\sigma} - 1} \abs{ \frac{\mu - \mu_0}{2 \sigma_0} } 
  \le \eps + \frac{\eps}{2} + \frac{\eps^2}{2} 
  \le 2 \eps. \qedhere
\end{align*} 
\end{proof}

\begin{lemma}[Lipschitz condition]
\label{lem:Lipschitz}
Let $m_\theta$ be as in \eqref{eq:m} with $p_\theta$ the GEV density function. For fixed $\theta_0 = (\gamma_0, \mu_0, \sigma_0) \in (-1/2, \infty) \times \reals \times (0, \infty)$, there exists $\dot{m}$ such that $P_{\theta_0} \dot{m}^2 < \infty$ and such that
\begin{align} \label{eq:Lip}
  \abs{ m_{\theta_1} - m_{\theta_2} } \, \ind_{ \{ p_{\theta_0} > 0 \} }
  \le \dot{m} \, \norm{ \theta_1 - \theta_2 }
\end{align}
for all $\theta_1$ and $\theta_2$ in a neighborhood of $\theta_0$.
\end{lemma}

\begin{proof}[Proof of Lemma~\ref{lem:Lipschitz}]
The function $\dot m$ can be constructed along the following lines. First, fix $a\in[1/2,1)$, to be determined later in terms of $\gamma_0$. Since $\abs{\log x - \log y} \le \abs{x-y} / \min(x,y)$ and since the map $z \mapsto \abs{ (x+z) ^a - (y+z)^a }$ is decreasing, we have, on $\{ x : p_{\theta_0}(x) > 0 \}$,
\begin{align*}
  \frac{1}{2} \abs{ m_{\theta_1} - m_{\theta_2} }
  &=  \abs{ \log ( p_{\theta_1} + p_{\theta_0} ) - \log ( p_{\theta_2} + p_{\theta_0} ) } \\
  &=  \frac{1}{a} \abs{ \log \{ (p_{\theta_1} + p_{\theta_0})^a \}  - \log \{ (p_{\theta_2} + p_{\theta_0})^a\} } \\
  &\le \frac{1}{a} \times  \frac{\abs{  (p_{\theta_1} + p_{\theta_0})^a   - (p_{\theta_2} + p_{\theta_0})^a } }%
  {\min \{ ( p_{\theta_1} + p_{\theta_0} )^a,( p_{\theta_2} + p_{\theta_0} )^a \} } \\
  &\le \frac{1}{a}  \times \frac{ \abs{ p_{\theta_1}^a - p_{\theta_2}^a } }{ p_{\theta_0}^a }.
\end{align*}
Suppose we can find a nonnegative function $\dot{p}_a$  such that, for some neighbourhood $V$ of $\theta_0$, we have, for each $k \in \{1, 2, 3\}$,
\begin{equation}
\label{eq:dots}
  \sup_{\theta \in V} \abs{ \partial_{\theta_k} p_\theta^a(x) }
  \le \dot{p}_a(x), \qquad x \in \{ p_{\theta_0} > 0 \}.
\end{equation}
Then we find, on $\{ p_{\theta_0} > 0 \}$,
\[
  \frac{1}{2} \abs{ m_{\theta_1} - m_{\theta_2} }
  \le \frac{3}{a} \, \frac{\dot p_a}{p_{\theta_0}^a} \norm{ \theta_1 - \theta_2 }.
\]
Hence, the Lipschitz condition \eqref{eq:Lip} is satisfied, with $\dot m = (6/a) \, \dot{p}_a \, p_{\theta_0}^{-a}$, provided that
\begin{align} 
\label{eq:bound}
  P_{\theta_0} [ (\dot{p}_a \, p_{\theta_0}^{-a})^2 ]
  = \int_{ \{ p_{\theta_0} > 0 \} } \dot{p}_a^2 (x) \, p_{\theta_0}^{1-2a}(x) \, \diff x
  < \infty.
\end{align}

We will split the domain $\{p_{\theta_0}>0\}$ into certain intervals and we will use the previous construction on each of these intervals separately, with possibly different values of $a$. 
For bounded intervals, a simplication may occur. Recall Lemma~\ref{lem:pdf:pow} and let $a \in [1/2, 1)$ be large enough such that $\gamma_0 > -a/(1+a)$. Let $V$ be a compact neighbourhood of $\theta_0$ within $(-a/(1+a), \infty) \times \reals \times (0, \infty)$ and let $I$ be a bounded interval. Since continuous functions are uniformly bounded on compacta, we have
\begin{equation}
\label{eq:unibound}
  \sup_{(x, \theta) \in I \times V} \max_{k \in \{1, 2, 3\}} \abs{ \partial_{\theta_k} p_\theta^a(x) }
  < \infty.
\end{equation}
Hence, on any bounded interval of the support of $P_{\theta_0}$ on which $p_{\theta_0}^{1-2a}$ is integrable, we may choose $\dot{m}$ equal to a constant times $p_{\theta_0}^{-a}$. Such intervals need therefore not be looked into further. For $a=1/2$, a choice which will occur often, the condition that $p_{\theta_0}^{1-2a}$ is integrable on bounded intervals is trivially satisfied.

To further control the partial derivatives of $p_\theta^a$, we will use the identity
\begin{equation*}
 \partial_{\theta_k} p_\theta^a(x)
  = a \, p_\theta^a(x) \,\partial_{\theta_k} \ell_\theta(x),
  \qquad k \in \{1, 2, 3\}.
\end{equation*}
We will seek bounds for the functions $p_\theta^a$ and $\abs{ \partial_{\theta_k} \ell_\theta }$ separately.

To facilitate writing, let us say that positive functions $A$ and $B$ of $(x, \theta)$ satisfy
\[ 
  A \lesssim B
\] 
if there exists $c( \theta_0 ) > 0$ such that $A(x, \theta) \le c( \theta_0 ) \, B(x, \theta)$ for all $(x, \theta)$ in the proper range. Here $c( \theta_0 )$ is a positive constant whose value may depend on $\theta_0$. For instance, for all $\theta = (\gamma, \mu, \sigma)$ in a compact neighbourhood of $\theta_0$ within $\reals \times \reals \times (0, \infty)$, we have $1/\sigma \lesssim 1$. The relation $\lesssim$ is transitive (and reflexive, but not anti-symmetric) and behaves well under multiplication of positive quantities.

\medskip

\noindent
\textbf{I. Case ${\gamma_0} > 0$}.
Fix $\theta_0 = (\gamma_0, \mu_0, \sigma_0) \in (0, \infty) \times \R \times (0, \infty)$ and set $a=1/2$ in \eqref{eq:dots} and \eqref{eq:bound}.
Fix $\eps \in (0, 1/4]$ and consider the following neighbourhood $V_\eps$ of $\theta_0$:
\[
  V_\eps =
  [(1-\eps) \gamma_0, (1+\eps) \gamma_0] 
  \times [\mu_0 - \eps \sigma_0, \mu_0 + \eps \sigma_0]
  \times \left[ \frac{\sigma_0}{1+\eps}, \frac{\sigma_0}{1-\eps}\right].
\]

The support of $P_{\theta_0}$ is $(\mu_0 - \sigma_0/\gamma_0, \infty)$. In view of \eqref{eq:unibound}, it is sufficient to construct $\dot{p}_{1/2}(x)$ for $x \ge \mu_0 + 2 \sigma_0$, i.e., ${z_0} = (x-\mu_0)/\sigma_0 \ge 2$. 
For such $x$ and  for $\theta \in V_\eps$, we have $p_{\theta}(x)>0$ as well as $\abs{z / z_0 - 1} \le 2\eps$ by Lemma~\ref{lem:relerr}. In particular, $1 \le z_0/2 \le z \le 2z_0$.

Put $u = (1 + \gamma z)^{-1/\gamma}$; note that $0 < u < 1$. By Lemmas~\ref{lem:score:bounds:mu}, \ref{lem:score:bounds:sigma} and \ref{lem:score:bounds}, all three score functions $\partial_{\theta_k} \ell_\theta(x)$ can be bounded in absolute value by a multiple of $1 + \log(1/u)$, uniformly in $\theta \in V_\eps$ and for all $x \ge \mu_0 + 2 \sigma_0$. Further, the density can be bounded by
\[
  p_\theta(x) = \frac{1}{\sigma} e^{-u} u^{\gamma+1} \lesssim u^{\gamma+1}.
\]
Hence
\[
  p_\theta(x) \, \abs{ \partial_{\theta_k} \ell_\theta(x) }
  \lesssim
  (1 + \log(1/u)) u^{\gamma + 1}
  \lesssim
  u^{\gamma + 1/2}
  =
  (1+\gamma z)^{-1 - 1/{2\gamma}}.
\]
Since $z > z_0/2$ and $0 < \gamma_0/2 < \gamma < 2\gamma_0$, the supremum over $\theta \in V_\eps$ of the previous upper bound is easily seen to be integrable over $z_0 \ge 2$.
\medskip

\noindent
\textbf{II. Case $\gamma_0 \in (-1/2, 0)$.} 
Fix $\eps\in(0,1/12)$ sufficiently small such that 
$-1/2 < \gamma_0 - 2\eps < \gamma_0 + 2\eps < 0$. Consider the following neighbourhood $V_\eps$ of $\theta_0$:
\[
  V_\eps = 
  [\gamma_0 - \eps, \gamma_0 + \eps] \times 
  [\mu_0 - \eps \sigma_0, \mu_0 + \eps \sigma_0] \times 
  \left[ \frac{\sigma_0}{1+\eps}, \frac{\sigma_0}{1-\eps} \right].
\]

The support of $P_{\theta_0}$ is $(-\infty, \mu_0 + \sigma_0 / \abs{\gamma_0})$. We split this set into two intervals:
\[
  (-\infty,  \mu_0-2\sigma_0], \qquad
  (\mu_0 -2 \sigma_0, \mu_0 + \sigma_0/\abs{\gamma_0}).
\]
\medskip

\noindent\textit{II.1.\ The interval $(\mu_0 -2 \sigma_0, \mu_0 + \sigma_0/\abs{\gamma_0})$.}
We follow the construction leading to \eqref{eq:dots} and \eqref{eq:bound}. To this end, we choose $a =a(\gamma_0) \in (1/2,1)$ in such a way that 
\[
\frac{\abs{\gamma_0}}{1-\abs{\gamma_0}} < a < \frac{1}{2(1-\abs{\gamma_0})}.
\]
Define
\[
  \dot{p}_a (x) =   \sup_{\theta \in  V} \max_{k \in \{1, 2, 3\}} \abs{ \partial_{\theta_k} p_\theta^a(x) }.
\]
Our choice of $a$ entails that $\gamma_0 > -a / (1+a)$. Hence, in view of \eqref{eq:unibound}, the function $\dot{p}_a$ is bounded on the interval $(\mu_0 -2 \sigma_0, \mu_0 + \sigma_0/\abs{\gamma_0})$. We need to show that the integral in \eqref{eq:bound} is finite when we restrict the domain to  $(\mu_0 -2 \sigma_0, \mu_0 + \sigma_0/\abs{\gamma_0})$. It is then sufficient to show that the function $p_{\theta_0}^{1-2a}$ is integrable on that set. Write down the integral and make a change of variable $z_0=(x-\mu_0)/\sigma_0$ to obtain that
\begin{align*}
 \int_{\mu_0 -2 \sigma_0}^{\mu_0 +  \sigma_0/\abs{\gamma_0}} p_{\theta_0}^{1-2a}(x) \, dx  
=&\
\int_{-2}^{1/\abs{\gamma_0}} \sigma_0^{2a} \exp\left\{-(1-2a) (1+\gamma_0 z_0) ^{-1/\gamma_0} \right\} \\
& \hspace{4.5cm} 
\times (1+\gamma_0 z_0)^{(1-2a)(1/\abs{\gamma_0}-1)}\, dz_0 \\
\lesssim &\
\int_{-2}^{1/\abs{\gamma_0}} (1+\gamma_0 z_0)^{(1-2a)(1/\abs{\gamma_0}-1)} \, dz_0.
\end{align*}
The proportionality constant in the last inequality only depends on $\theta_0$.
The right-hand side is finite since the exponent is larger than $-1$ by our choice of $a$.

\medskip

\noindent\textit{II.2.\ The interval $(-\infty, \mu_0 - 2 \sigma_0]$.}
We construct $\dot{m}(x)$ for $x \le \mu_0 - 2 \sigma_0$, i.e., for $z_0 = (x - \mu_0)/\sigma_0 \le -2$. We will again use the construction leading to \eqref{eq:dots} and \eqref{eq:bound}, this time choosing $a=1/2$. In that case, it is sufficient to construct $\dot{p}_{1/2}$ such that it is square-integrable on $(-\infty, \mu_0 - 2 \sigma_0)$ with respect to the Lebesgue measure. 

For $x \le \mu_0-2\sigma_0$ and for $\theta \in V_\eps$ we have $p_\theta(x) > 0$ and, by Lemma~\ref{lem:relerr}, $\abs{z/z_0-1} \le 2\eps$ and therefore $2z_0 \le z \le z_0/2 \le -1$.

Since $z$ is negative, we have $u \ge 1$. The density satisfies
\begin{align*}
  p_\theta(x) 
&= \frac{1}{\sigma} e^{-u} u^{\gamma+1} 
\lesssim
  e^{-u} u.
\end{align*}
By Lemmas~\ref{lem:score:bounds:mu}, \ref{lem:score:bounds:sigma} and \ref{lem:score:bounds}, the three score functions $\partial_{\theta_k} \ell_\theta(x)$ can all be bounded by a multiple of $u^2$, the proportionality constant neither depending on $x \le \mu_0 - 2 \sigma_0$ nor on $\theta \in V_\eps$. Hence,  for all $k \in \{1, 2, 3\}$, all $\theta \in V_\eps$ and all $x \le \mu_0 - 2 \sigma_0$,
\begin{align*}
  p_\theta (x) \, \abs{ \partial_{\theta_k} \ell_\theta(x) }^2 
   \lesssim 
   e^{-u} u^{5}  \lesssim e^{-u/2}
\end{align*}
(using that $\sup_{u \ge 1} u^m e^{-u/2} < \infty$ for any scalar $m$). Since $\gamma \ge \gamma_0 - \eps$ and $ z\le z_0/2 < 0$, we have $u \ge \{ 1 + (\gamma_0 - \eps)  z_0/2 \}^{-1/(\gamma_0 - \eps)}$. Inserting this bound into the last display yields a function which is integrable over $z_0 \in (-\infty, -2)$.
\medskip

\noindent
\textbf{III. Case $ {\gamma_0} = 0$}.
For fixed $\eps \in (0, 1/6]$, consider the following compact neighbourhood of $\theta_0 = (0, \mu_0, \sigma_0)$:
\[
  V_\eps = 
  [-\eps, \eps] \times 
  [ \mu_0 - \sigma_0 \eps, \mu_0 + \sigma_0 \eps ] \times 
  \left[ \frac{\sigma_0}{1+\eps}, \frac{\sigma_0}{1-\eps} \right].
\]
Partition this set in two pieces, according to the sign of $\gamma$:
\begin{align*}
  V_{\eps,+} &= \{ \theta \in V_\eps : \gamma \ge 0 \},  \qquad
  V_{\eps,-} = \{ \theta \in V_\eps : \gamma \le 0 \}.
\end{align*}
The support of the Gumbel distribution is $\R$, which we will decompose into three intervals:
\[
  (-\infty, \mu_0 - \sigma_0 / \eps ], \qquad
  [\mu_0 - \sigma_0/\eps, \mu_0 + \sigma_0 / \eps], \qquad
  [\mu_0 + \sigma_0 / \eps, \infty).
\]
The middle interval is bounded. By \eqref{eq:unibound} with $a=1/2$, we only need to consider the cases $ {z_0} \ge 1/\eps$ and $ {z_0} \le -1/\eps$, where $ {z_0} = (x - \mu_0) / \sigma_0$. Put $ {z} = (x - \mu) / \sigma$ and note that $\abs{z / z_0 - 1} \le 2 \eps$ by Lemma~\ref{lem:relerr}.
\medskip

\noindent
\textit{III.1.\ Case $ {z_0} \ge 1/\eps$}.
We have $ {z_0} \ge 1/\eps \ge 6$ and $4 \le (1-2\eps)  {z_0} \le  {z} \le (1+2\eps)  {z_0}$. Moreover, $0 < u < 1$.
We will write the supremum over $\theta \in V_\eps$ as the maximum of the suprema over $\theta \in V_{\eps,+}$ and $\theta \in V_{\eps,-}$.
\medskip

\noindent
\textit{III.1.1.\ Case $\theta \in V_{\eps,+}$.}
In this case, always $1 + \gamma  {z} >  {1}$. The bounds on the scores in Lemmas~\ref{lem:score:bounds:mu}, \ref{lem:score:bounds:sigma}, and \ref{lem:score:bounds} imply that\begin{align*}
  \abs{ \partial_\gamma \ell_\theta(x) }
  &\lesssim  {z}^2 \lesssim  {z_0}^2, \qquad
  \abs{ \partial_\mu \ell_\theta(x) }
  \lesssim 1, \qquad
  \abs{ \partial_\sigma \ell_\theta(x) }
  \lesssim  {z} \lesssim  {z_0}.
\end{align*}
Further, we have
\begin{align*}
  p_\theta(x) 
  &\lesssim (1 + \gamma  {z})^{-1/\gamma - 1} 
  \le (1 + \gamma  {z})^{-1/\gamma} 
  \le (1 + \eps  {z})^{-1/\eps} 
  \le \{ 1 + \eps (1-2\eps)  {z_0} \}^{-6}.
\end{align*}
It follows that, for each $k \in \{1, 2, 3\}$,
\[
  \sup_{\theta \in V_{\eps,+}} p_\theta(x) \abs{ \partial_{\theta_k} \ell_\theta(x) }^2
  \lesssim \{ 1 + \eps (1-2\eps)  {z_0} \}^{-6} z_0^4.
\]
The right-hand side is integrable over $ {z_0} \in [1/\eps, \infty)$.

\medskip
\noindent
\textit{III.1.2.\ Case $\theta \in V_{\eps,-}$.} 
If $-\eps \le \gamma \le -1/ {z}$, then $1 + \gamma  {z} \le 0$ and thus $p_\theta(x) = 0$. So suppose $\gamma > - 1/ {z}$. By Lemmas~\ref{lem:score:bounds:mu}, \ref{lem:score:bounds:sigma}, and \ref{lem:score:bounds}, the scores can be bounded as follows:
\begin{equation}
\label{eq:gamma-:z+}
  \max_{k \in \{1, 2, 3\}} \abs{ \partial_{\theta_k} \ell_\theta(x) } \lesssim \frac{ {z}^2}{1 + \gamma  {z}}.
\end{equation}
Moreover, the density is bounded by
\[
  p_\theta(x) \lesssim (1 + \gamma  {z})^{-1/\gamma - 1}.
\]
Hence, for $k \in \{1, 2, 3\}$, since $\gamma \mapsto (1 + \gamma  {z})^{-1/\gamma}$ is increasing in $\gamma$  {on $\{\gamma : 1+\gamma z\}$},
\begin{align*}
  \sup_{\theta \in V_{\eps,+}} p_\theta(x) \, \abs{ \partial_{\theta_k} \ell_\theta(x) }^2
  &\lesssim  {z}^4 \, (1 + \gamma  {z})^{-1/\gamma - 3} \\
  &=  {z}^4 \, \{ (1 + \gamma  {z})^{-1/\gamma} \}^{1 + 3\gamma} 
  \le  {z}^4 \, e^{-(1 + 3\gamma)  {z}} 
  \le  {z}^4 \, e^{- {z}/2} 
  \lesssim e^{- {z}/4},
\end{align*}
as $\sup_{ {z} \ge 0}  {z}^4 e^{- {z}/4} < \infty$. Bounding $ {z}$ in $e^{-z/4}$ from below by $ {z_0} / 2$ yields a function which is integrable over $ {z_0} \in [1/\eps, \infty)$.

\medskip
\noindent
\textit{III.2.\ Case $ {z_0} \le - 1/\eps$.}
We have $ {z_0} \le -1/\eps \le -6$ and, by Lemma~\ref{lem:relerr}, $(1+2\eps) {z_0} \le  {z} \le (1-2\eps) {z_0}  \le -4$. Moreover, $u > 1$.

\medskip
\noindent
\textit{III.2.1.\ Case $\theta \in V_{\eps,+}$.}
If $\gamma \ge 1 / \abs{ {z}}$, then $p_\theta(x) = 0$. Assume $0 \le \gamma < 1 / \abs{ {z}}$, so that $1 + \gamma  {z} > 0$. By \eqref{eq:scoregamma:+-},
\begin{align} \label{eq:gamma0+-1}
  \abs{ \partial_\gamma \ell_\theta(x) } \leq u^{1+\gamma} \{ (\log u)^2 + \log u) \} \lesssim u^{1 + 2\eps}.
\end{align}
Further, by Lemmas~\ref{lem:score:bounds:mu} and \ref{lem:score:bounds:sigma},
\begin{align}\label{eq:gamma0+-2}
  \abs{ \partial_\mu \ell_\theta(x) } &\lesssim u^{1+\gamma} \le u^{1+\eps}, \qquad
  \abs{ \partial_\sigma \ell_\theta(x) } \lesssim u^{1 + \gamma} \{ \log(u) + 1 \} \lesssim u^{1 + 2\eps}.
\end{align}
The density is bounded by
\[
  p_\theta(x) \lesssim e^{-u} u^{1+\gamma}.
\]
In total, since $\sup_{u \ge 1} u^m e^{-u/2} < \infty$ for any scalar $m$,
\[
  p_\theta(x) \, \abs{ \partial_{\theta_k} \ell_\theta(x) }^2
  \lesssim e^{-u} u^{(1 + \gamma) + 2(1 + 2\eps)}
  \lesssim e^{-u} u^{3 + 5\eps}
  \lesssim e^{-u/2}.
\]
Since $\gamma \mapsto (1 + \gamma  {z})^{-1/\gamma}$ is increasing in $\gamma$, a lower bound for $u$ is given by $u \ge e^{- {z}} = e^{\abs{ {z}}} \ge e^{\abs{ {z_0}}/2}$. Plugging this bound into $e^{-u/2}$ yields a function which is integrable over $ {z_0} \in (-\infty, -1/\eps]$.

\medskip
\noindent
\textit{III.2.2.\ Case $\theta \in V_{\eps,-}$.}
By Lemmas~\ref{lem:score:bounds:mu}, \ref{lem:score:bounds:sigma} and \ref{lem:score:bounds},
\begin{align*}
  \abs{ \partial_\gamma \ell_\theta(x) } \lesssim u^{1+\eps}, \qquad
  \abs{ \partial_\mu \ell_\theta(x) } &\lesssim u^{ {1+\eps}},  \qquad
  \abs{ \partial_\sigma \ell_\theta(x) } \lesssim u \, \log(u).
\end{align*}
The density is bounded by
\[
  p_\theta(x) \lesssim e^{-u} \, u^{1+\gamma}.
\]
We get
\[
  \sup_{\theta \in V_{\eps,-}} p_\theta(x) \, \abs{ \partial_{\theta_k} \ell_\theta(x) }^2
  \lesssim e^{-u} \, u^{3 + 3\eps} \lesssim e^{-u/2}.
\]
Further, $u$ can be bounded from below by $( 1 + (-\eps) (1-2\eps) z_0 )^{-1/(-\eps)}$. Inserting this into $e^{-u/2}$ yields an integrable function in $ {z_0} \in (-\infty, -1/\eps]$.
\end{proof}

\begin{lemma}
\label{lem:second_order_expansion}
Let $P_\theta$ denote the three-parameter GEV distribution with parameter $\theta$ and let $m_\theta$ be as in \eqref{eq:m}. For fixed $\theta_0$, the map $\theta \mapsto P_{\theta_0} m_\theta$ attains a unique maximum over $(-1/2, \infty) \times \R \times (0,\infty)$ at $\theta = \theta_0$. If moreover $\gamma_0 > -1/2$, then whenever $h_t \to h$ in $\reals^k$ as $t \to 0$, we have
\[
  t^{-2} P_{\theta_0} \bigl( m_{\theta_0 + t h_t} - m_{\theta_0} \bigr)
  \to
  - \frac{1}{4} h^T I_{\theta_0} h,
  \qquad t \to 0.
\]
\end{lemma}

\begin{proof}
The fact that $\theta_0$ is the unique maximizer of the map $\theta \mapsto P_{\theta_0}$ follows from Lemma~5.35 in \cite{vdV98} applied to the mixture densities $(p_\theta + p_{\theta_0})/2$.

The second-order Taylor expansion can now be shown along similar lines as in the proof of Theo\-rem~5.39 in \cite{vdV98}. In fact, the same technique was used to show \eqref{eq:2ndd} in the proof of Proposition~\ref{prop:mlelip}.
\end{proof}

\section{Remaining steps for the proof of Proposition~\ref{prop:an}}
\label{app:an}

In view of the outline of the proof given right after the statement of the Proposition~\ref{prop:an}, it remains to check the condition on the support in \eqref{eq:suppeps} and the Lipschitz property \eqref{eq:lip}. This is the content of the present section.

\begin{lemma}
\label{lem:suppgev}
For any $\theta_0 \in (-1/2, \infty) \times \R \times (0,\infty)$, the three-parameter GEV family satisfies Condition~\eqref{eq:suppeps}.
\end{lemma}

\begin{proof}
First, consider $\theta_0\in\nnd= (-\infty, 0) \times \R \times (0,\infty)$ such that $\gamma_0> -1/2$.  Let $\omega(\theta) = \mu - \sigma / \gamma = \mu + \sigma / \abs{\gamma}$ and $\omega_0 = \omega(\theta_0)$. Since $\theta \mapsto \omega(\theta)$ is increasing in each component of $\theta \in \nnd$, we have
\begin{align*} 
  h_-(\eps) :=  \inf \{ \omega(\theta) : \theta \in U_\eps(\theta_0) \}
  =
  \omega(\gamma_0-\eps, \mu_0-\eps, \sigma_0-\eps)
\end{align*}
whence $\bar S(\eps) = (- \infty, h_-(\eps))$. The function $t \mapsto h_-(t)$ is continuously differentiable on a neighbourhood of $t = 0$ and with negative derivative at $t = 0$. Hence, we can find constants $0<b_1<c_1$ and $t_1>0$ such that
\begin{equation}
\label{eq:omega-}
  \omega_0 - t \cdot c_1 \sigma_0/\abs{\gamma_0} \le h_-(t)  \le \omega_0 - t \cdot b_1 \sigma_0/\abs{\gamma_0},
  \qquad t \in [0, t_1].
\end{equation}
As a consequence of \eqref{eq:omega-}, for sufficiently small $\eps>0$,
\[
\Xc \setminus \bar S(\eps) \subset [ \omega_0   -  \eps c_1 \sigma_0 / \abs{\gamma_0}, \infty).
\]
Therefore, using the substitution $u = \{1+\gamma_0 (x - \mu_0)/\sigma_0\}^{1/\abs{\gamma_0}}$,
\begin{align*}
P_{\theta_0} (\Xc \setminus \bar S(\eps) ) 
\le
\int_{ \omega_0   - c_1 \eps \sigma_0 / \abs{\gamma_0}}^{\omega_0} p_{\theta_0}(x) \, dx
=
\int_0^{(\eps c_1)^{1/\abs{\gamma_0}}} e^{-u} \, du =O( \eps^{1/\abs{\gamma_0}}) = o(\eps^2)
\end{align*}
as $\eps\downarrow 0$, since $-1/2 < {\gamma_0}<1$.

Now, consider $\theta_0 \in\pnd=(0,\infty) \times \R \times (0,\infty)$, i.e.,  $\gamma_0 > 0$. Then, for any $\eps < \min(\sigma_0, \gamma_0)$ and any $\theta\in U_\eps(\theta_0)$, we have
\[
\omega(\theta) = \mu - \frac{\sigma}{\gamma} \le \mu_0 + \eps - \frac{\sigma_0 - \eps} {\gamma_0 + \eps} =: h_+(\eps).
\]
whence $\bar S(\eps)=(h_+(\eps), \infty)$. The function $t \mapsto h_+(\eps)$ is continuously differentiable on a neighbourhood of $t=0$, with positive derivative at $t=0$. As a consequence, we can find constants $0<b_2<c_2$ and $t_2>0$ such that
\begin{align} \label{eq:omega+}
\omega_0 +  t \cdot b_2 \sigma_0 /\gamma_0 \le h_+(t) \le \omega_0 + t \cdot c_2 \sigma_0 /\gamma_0 , \qquad t \in [0, t_2]. 
\end{align}
Hence, for sufficiently small $\eps>0$, we obtain that
\[
\Xc \setminus \bar S(\eps) \subset (- \infty, \omega_0 +  \eps  \cdot c_2 \sigma_0 /\gamma_0].
\]
Therefore, using the substitution $u = \{1+\gamma_0 (x - \mu_0)/\sigma_0\}^{-1/\gamma_0}$, 
\begin{align*}
P_{\theta_0} (\Xc \setminus \bar S(\eps) ) 
\le
\int_{ \omega_0}^{\omega_0 + \eps c_2 \sigma_0 /\gamma_0} p_{\theta_0}(x) \, dx
=
\int_{(c_2\eps)^{-1/\gamma_0}}^\infty e^{-u} \, du = \exp \{ - (c_2\eps)^{-1/\gamma_0} \},
\end{align*}
which clearly is of the order $o(\eps^2)$ as $\eps\downarrow 0$, since $\gamma_0 > 0$.

Finally, consider $\theta_0 \in \nd$ with $\gamma_0 = 0$. Then, for any $\eps<\sigma_0/2$, we have
\[
\omega(\theta) = \mu - \frac\sigma\gamma  \le \mu_0 + \eps - \frac{\sigma_0 - \eps}{\eps} =: h_{0,+}(\eps), \qquad \theta \in U_\eps(\theta_0 ) \cap \pnd,
\]
as well as
\[
\omega(\theta) = \mu + \frac\sigma{\abs{\gamma}}  \ge \mu_0 - \eps + \frac{\sigma_0 - \eps}{\eps} =: h_{0,-}(\eps), \qquad \theta \in U_\eps(\theta_0 ) \cap \nnd.
\]
As a consequence, for sufficiently small $\eps$, $\bar S(\eps)= (h_{0,+}(\eps), h_{0,-}(\eps))$. Clearly, we can find constants $d_1, d_2 >0 $ such that $h_{0,+}(\eps) \le -d_1/\eps$  and $h_{0,-}(\eps) \ge d_2/\eps$ for all sufficiently small $\eps$. Hence,
\begin{align*}
P_{\theta_0} (\Xc \setminus \bar S(\eps) ) 
&\le
\int_{-\infty}^{-d_1/\eps} p_{\theta_0}(x) \, dx + \int_{d_2/\eps}^\infty p_{\theta_0}(x) \, dx  \\
&= \exp\left\{ -\exp \left( \frac{d_1}{\sigma_0 \eps} + \frac{\mu_0}{\sigma_0} \right) \right\}  
	+ 1 -  \exp\left\{ -\exp \left( - \frac{d_2}{\sigma_0 \eps} + \frac{\mu_0}{\sigma_0} \right) \right\},
\end{align*}
which can be easily seen to be of the order $o(\eps^2)$ as $\eps \downarrow 0$, since $1 - \exp(-y) = (1 + o(1)) y$ as $y \to 0$.
\end{proof}

\begin{lemma}
\label{lem:lipgev}
For any $\theta_0 \in (-1/2, \infty) \times \R \times (0,\infty)$, the three-parameter GEV family satisfies Condition~\eqref{eq:lip} with $\Theta = (-1/2,\infty) \times \R \times (0,\infty)$ (and hence also with any compact subset  containing $\theta_0$ in its interior). 
\end{lemma}

\begin{proof} 
Throughout, let $\| \cdot \|$ denote the maximum norm on $\R^3$. The proof will be split up in three cases, according to the sign of $\gamma_0$. For some suitable $\eps_0>0$ to be specified below, and for any $x \in S_{\theta_0}$, define a set of \textit{admissible parameter vectors} $\theta$  as
\begin{equation*}
  \Theta_0(x)
  =
  \bigl\{
    \theta \in U_{\eps_0}(\theta_0) :
    \text{if $\theta' \in \Theta$ satisfies $\norm{\theta' - \theta_0} \le 2 \norm{\theta - \theta_0}$, then $x \in S_{\theta'}$}
  \bigr\}.
\end{equation*}
Note that, if $\theta \in \Theta_0(x)$, then the entire ball with center $\theta_0$ and radius $\norm{ \theta-\theta_0 }$ is a subset of $\Theta_0(x)$.
In other words, $\Theta_0(x)$ is the union of all neighbourhoods $U_\eps(\theta_0)$ for those $\eps \in (0, \eps_0]$ such that $x \in S_{\theta'}$ for every $\theta' \in U_{2 \eps}(\theta_0)$. 
 [The multiplicative constant $2$ is not essential and could have been replaced by an arbitrary constant $c_0 > 1$.]

It follows from the definition of $\Theta_0(x)$ that $\theta \mapsto \ell_\theta(x)$ is continuously differentiable on $\Theta_0(x)$. Hence, we may define
\begin{equation}
\label{eq:dotell2}
  \dot{\ell}(x)
  =
  3 \sup \{ \norm{ \dot{\ell}_\theta(x) } : \theta \in \Theta_0(x) \}, \qquad x \in S_{\theta_0},
\end{equation}
and, by the mean-value theorem, immediately obtain that
\begin{equation}
\label{eq:Lipschitz2}
  \forall x \in S_{\theta_0} : 
  \forall \theta_1, \theta_2 \in \Theta_0(x) : \quad
  \abs{ \ell_{\theta_1}(x) - \ell_{\theta_2}(x) }
  \le
  \dot{\ell}(x) \, \norm{ \theta_1 - \theta_2 }.
\end{equation}
[The constant $3$ appears because of the use of the max-norm.] It can be seen easily that \eqref{eq:Lipschitz2} implies \eqref{eq:lip}. The proof will be finished once we will have constructed $\eps_0$ and will have showed that the function in \eqref{eq:dotell2} is square-integrable with respect to $P_{\theta_0}$. 

\medskip
\noindent
\textbf{I. Case $\gamma_0 \in (0, \infty)$.}
Write $\omega_0 = \mu_0 - \sigma_0 / \gamma_0$ and recall that $S_{\theta_0} = (\omega_0, \infty)$. For $x \in S_{\theta_0}$, write $z_0 = z_0(x) = (x - \mu_0) / \sigma_0$. Note that $z_0 > - 1 / \gamma_0$.

For a parameter vector $\theta$ with $\gamma > 0$, we have $S_\theta = (\omega(\theta), \infty)$ where $\omega(\theta) = \mu - \sigma/\gamma$. A monotonicity and differentiability argument similar to that yielding  \eqref{eq:omega+} shows that there exist positive constants $b_0$ and $t_0$ with $b_0 t_0 \le 1/2$ and $t_0 < \min(\sigma_0, \gamma_0)/2$ such that for all $t \in [0, t_0]$,
\[
  h_+(t)
  =
  \sup_{\theta \in U_t(\theta_0)} \omega(\theta)
  =
  \mu_0 + t - \frac{\sigma_0 - t}{\gamma_0 + t}
  \quad
  \left\{
    \begin{array}{rl}
      \le& \omega_0 + (\sigma_0/\gamma_0) \, t b_0 \, (4/3), \\[1em]
      \ge& \omega_0 + (\sigma_0/\gamma_0) \, t b_0.
    \end{array}
  \right.
\]
Partition the interval $S_{\theta_0}$ into three sub-intervals:
\[
  S_{\theta_0} 
  = 
  (\omega_0, \; \omega_0 +  (\sigma_0/\gamma_0) \, t_0 b_0)
  \cup 
  [\omega_0 + (\sigma_0/\gamma_0) \, t_0 b_0 , \; \mu_0 + 2 \sigma_0]
  \cup 
  (\mu_0 + 2 \sigma_0, \; \infty).
\]

Choose $0 < \eps_0 < \min \{ 1/4, t_0/2, \sigma_0/(2\gamma_0) \}$ small enough such that
\begin{align*}
  \sigma / \sigma_0, \gamma / \gamma_0 \in [5/6, 7/6], \quad
  (\sigma/\gamma)/(\sigma_0/\gamma_0) \in [3/4, 4/3],
  \qquad \text{for all $\theta \in U_{\eps_0}(\theta_0)$.}
\end{align*}
We will provide an upper bound for $\dot{\ell}(x)$ in \eqref{eq:dotell2} for each $x$ in each piece of the partition separately. 

First, suppose that $x \in S_{\theta_0}$ is such that $x = \omega_0 + (\sigma_0/\gamma_0) y$ with $0 < y < t_0 b_0$.  If $\theta$ is such that $t := \norm{\theta - \theta_0}$ satisfies $t < \eps_0 \le  t_0/2$ and $2t b_0 \ge y$, then $h_+(2t) \ge \omega_0 + (\sigma_0 / \gamma_0) (2tb_0) \ge x$ and thus $\theta \not\in \Theta_0(x)$. As a consequence, $\Theta_0(x) \subset U_{y/(2b_0)}(\theta_0)$. But for $\theta \in U_{y/(2b_0)}(\theta_0)$, we have $\omega(\theta) \le \omega_0 + (\sigma_0/\gamma_0) (y/2) (4/3) = \omega_0 +  (\sigma_0/\gamma_0)(2/3)y$ and thus $x - \omega(\theta) \ge (\sigma_0 / \gamma_0) (1/3) y$. For such $\theta$, writing $z = (x-\mu)/\sigma$, we find, using $y = 1 + \gamma_0 z_0$, that $1 + \gamma z = (\gamma/\sigma) (x - \omega(\theta)) \ge (\gamma/\sigma) (\sigma_0/\gamma_0) (1/3) y \ge (1/4) (1 + \gamma_0 z_0)$. In addition, $x < \omega_0 + (\sigma_0/\gamma_0) (1/2) = \mu_0 - (\sigma_0/\gamma_0) (1/2) < \mu_0 - \eps_0 \le \mu$ and thus $z < 0$. Apply \eqref{eq:scoremu}, \eqref{eq:scoresigma} and \eqref{eq:scoregamma:+-} to arrive at a $P_{\theta_0}$ square-integrable bound.

Second, suppose that $x \in S_{\theta_0}$ is such that $\omega_0 + (\sigma_0/\gamma_0) \, t_0 b_0 \le x \le \mu_0 + 2 \sigma_0$. The partial derivatives of $\ell_\theta(x)$ with respect to the three components of $\theta$ being continuous functions of $(\theta, x)$ in the compact domain $\{ \theta : \norm{\theta - \theta_0} \le \eps_0 \} \times [\omega_0 + (\sigma_0/\gamma_0) \, t_0 b_0 , \; \mu_0 + 2 \sigma_0]$, they are also uniformly bounded and thus square-integrable with respect to $P_{\theta_0}$.

Third, suppose that $x \in S_{\theta_0}$ is such that $z_0 \ge 2$. Let $\theta \in U_{\eps_0}(\theta_0)$ and write $z = (x - \mu)/\sigma$. By Lemma~\ref{lem:relerr} we have $\abs{(z/z_0) - 1} \le 1/2$ and thus $z_0/2 \le z \le 3z_0/2$. The inequalities in Lemmas~\ref{lem:score:bounds:mu}, \ref{lem:score:bounds:sigma} and \ref{lem:score:bounds} then combine into a $P_{\theta_0}$ square-integrable upper bound for $\dot{\ell}$.

\medskip

\noindent
\textbf{II. Case $\gamma_0\in (-1/2,0)$.}
Write $\omega_0 = \mu_0 + \sigma_0 / \abs{\gamma_0}$ and recall that $S_{\theta_0} = (- \infty, \omega_0)$. For $x \in S_{\theta_0}$, write $z_0 = z_0(x) = (x - \mu_0) / \sigma_0$. Note that $z_0 <  1 / \abs{\gamma_0}$.

For a parameter vector $\theta$ with $\gamma < 0$, we have $S_\theta = (-\infty, \omega(\theta))$ where $\omega(\theta) = \mu + \sigma/\abs{\gamma}$. A monotonicity and differentiability argument similar to that yielding  \eqref{eq:omega-} shows that there exist positive constants $b_0$ and $t_0$ with $b_0 t_0 \le 1/2$ and $t_0 < \min(\sigma_0, \abs{\gamma_0})/2$ such that for all $t \in [0, t_0]$,
\[
  h_-(t)
  =
  \inf_{\theta \in U_t(\theta_0)} \omega(\theta)
  =
  \mu_0 - t - \frac{\sigma_0 - t}{\gamma_0 - t}
  \quad
  \left\{
    \begin{array}{rl}
      \le& \omega_0 - (\sigma_0/\abs{\gamma_0}) \, t b_0, \\[1em]
      \ge& \omega_0 - (\sigma_0/\abs{\gamma_0}) \, t b_0  \, (4/3).
    \end{array}
  \right.
\]
Partition the interval $S_{\theta_0}$ into three sub-intervals:
\[
S_{\theta_0} 
= 
(-\infty, \mu_0 - 2 \sigma_0) \cup
[\mu_0 - 2 \sigma_0, \omega_0 - (\sigma_0/\abs{\gamma_0}) t_0 b_0] \cup
( \omega_0 - (\sigma_0/\abs{\gamma_0}) t_0 b_0, \omega_0).
\]

Choose $0 < \eps_0 < \min \{ 1/4, t_0/2, \sigma_0/(2\abs{\gamma}_0), \sigma_0/7 \}$ small enough such that
\begin{align*}
  \sigma / \sigma_0, \gamma / \gamma_0 \in [5/6, 7/6], \quad
  (\sigma/\gamma)/(\sigma_0/\gamma_0) \in [3/4, 4/3],
  \qquad \text{for all $\theta \in U_{\eps_0}(\theta_0)$.}
\end{align*}
We will provide an upper bound for $\dot{\ell}(x)$ in \eqref{eq:dotell2} for each $x$ in each piece of the partition separately.

First, consider  $x \in (\omega_0 - (\sigma_0/\abs{\gamma_0}) t_0 b_0, \omega_0)$, which can be rewritten as 
$x=\omega_0 - (\sigma_0/\abs{\gamma_0}) y$ for some $0<y<t_0b_0$. Then $\Theta_0(x) \subset U_{y/(2b_0)}(\theta_0)$, for, if $\theta \in U_{\eps_0}(\theta_0) \subset U_{t_0/2}(\theta_0)$  satisfies $t=\|\theta-\theta_0\| \ge y/(2b_0)$, then there exists $\theta'$ with $\|\theta' - \theta_0\| \le  2 t$ such that $\omega(\theta') = h_-(2t)  \le \omega_0 - (\sigma_0/\abs{\gamma_0}) 2 t b_0 \le x$. Now, for $\theta \in U_{y/(2b_0)}(\theta_0)$, we have
$
\omega(\theta) - x \ge \omega_0 - \sigma_0/\abs{\gamma_0} (y/2)(4/3) -  x = \omega_0 - (2/3) (\omega_0 - x) - x = (1/3)(\omega_0 - x) .
$
In addition, $x> \omega_0 - \sigma_0/\abs{\gamma_0}/2 = \mu_0 + \sigma_0/\abs{\gamma_0} > \mu_0 + \eps_0 > \mu$, whence $z=(x-\mu)/\sigma>0$. 
Therefore, as a consequence of Lemmas~\ref{lem:score:bounds:mu}, \ref{lem:score:bounds:sigma} and \ref{lem:score:bounds}, we obtain the upper bound
\[
\| \dot \ell_\theta(x) \| \lesssim \frac{1}{1+\gamma z} =    \frac{\sigma}{\abs{\gamma}} \frac{1}{\omega(\theta) - x}
\le  \frac{(4/3) \sigma_0 }{(1/3) \abs{\gamma_0}} \, \frac{1}{\omega_0 - x}.
\]
Since $\gamma_0 \in (-1/2, 0)$ implies $1/\abs{\gamma_0}> 2$, the bound can be seen to be square-integrable with respect to $P_{\theta_0}$.

Second, consider $x \in [\mu_0 - 2 \sigma_0, \omega_0 - (\sigma_0/\abs{\gamma_0}) t_0 b_0]$.
The partial derivatives of $\ell_\theta(x)$ with respect to the three components of $\theta$ being continuous functions of $(\theta, x)$ in the compact domain $\{ \theta : \norm{\theta - \theta_0} \le \eps_0 \} \times  [\mu_0 - 2 \sigma_0, \omega_0 - (\sigma_0/\abs{\gamma_0}) t_0 b_0]$, they are also uniformly bounded and thus square-integrable with respect to $P_{\theta_0}$.

Third, consider $x \in (-\infty, \mu_0 - 2 \sigma_0)$. Let $\theta \in U_{\eps_0}(\theta_0)$ and write $z = (x - \mu)/\sigma$. By Lemma~\ref{lem:relerr}, we have $\abs{(z/z_0) - 1} \le 1/2$ and thus $3z_0/2 \le z \le z_0/2 \le -1$. The inequalities in Lemmas~\ref{lem:score:bounds:mu}, \ref{lem:score:bounds:sigma} and \ref{lem:score:bounds} then combine into a $P_{\theta_0}$ square-integrable upper bound for $\dot{\ell}$.

\medskip
\noindent
\textbf{III. Case $\gamma_0 = 0$.}   Recall that $S_{\theta_0} = \R$. Find $t_0 > 0$ small enough such that $(1+t_0) t_0 \le \sigma_0/4$. Then we have
\[
  (3/4) \sigma_0/t \le (\sigma_0/t) - t - 1 \le \sigma_0/t, \qquad \text{for $t \in (0, t_0]$.}
\]
It follows that, for all $t \in (0, t_0]$,
\begin{align}
\label{eq:h0+}
h_{0,+}(t) 
&= \sup_{\theta \in U_{t}(\theta_0) \cap \pnd} \omega(\theta) 
= \mu_0 + t - \frac{\sigma_0 - t}{t}
  \quad
  \left\{
    \begin{array}{rl}
      \le&  \mu_0 - (3/4) \sigma_0 / t,\\
      \ge& \mu_0 - \sigma_0/ t,
\end{array}
\right. \\
\label{eq:h0-}
h_{0,-}(t) 
&= \inf_{\theta \in U_{t}(\theta_0) \cap \nnd} \omega(\theta) 
= \mu_0 - t + \frac{\sigma_0 - t}{t}
  \quad
  \left\{
    \begin{array}{rl}
      \le&  \mu_0 + \sigma_0 / t,\\
      \ge&  \mu_0 + (3/4) \sigma_0/ t.
\end{array}
\right.
\end{align}
Choose $0 < \eps_0 < \min \{ 1/4, t_0/2 \}$ small enough such that $ \sigma / \sigma_0 \in [5/6, 7/6]$ whenever $\theta \in U_{\eps_0}(\theta_0)$. Partition the real line into three sub-intervals:
\[
\R
= 
(-\infty, \mu_0 - \sigma_0/\eps_0) \cup
[\mu_0 - \sigma_0/\eps_0, \mu_0 + \sigma_0/\eps_0] \cup
( \mu_0 + \sigma_0/\eps_0, \infty).
\]
We will provide an upper bound for $\dot{\ell}(x)$ in \eqref{eq:dotell2} for each $x$ in each piece of the partition separately. Write $x = \mu_0 + \sigma_0 z_0$.

\smallskip
\noindent
\textit{III.1 -- Case $x<\mu_0 - \sigma_0/\eps_0$.} 
We have $z_0 < -1/\eps_0$. For any $\theta \in \Theta_0(x)  \subset U_{\eps_0}(\theta_0)$, we have $x < \mu_0 - \sigma_0 / \eps_0 < \mu_0 - \eps_0 \le \mu$; note that $\eps_0^2 < t_0^2 < \sigma_0$. Therefore, $z = z_\theta(x) = (x - \mu)/\sigma < 0$ and thus $u = u_\gamma(z) > 1$, see Figure~\ref{fig:ugammaz}.

We claim that $\Theta_0(x) \subset U_{1/(2\abs{z_0})}(\theta_0)$. To prove this, it suffices to show that we can find a point $\theta'$ such that $\norm{\theta' - \theta_0} = 2 / (2\abs{z_0}) = 1 / \abs{z_0}$ and such that $x \notin S_{\theta'}$. But this is easy: just set $\theta' = (1/\abs{z_0}, \mu_0, \sigma_0)$ and note that $1 + (1/\abs{z_0}) z_0 = 1 - 1 = 0$.

Let $\theta \in \Theta_0(x) \subset U_{1/(2\abs{z_0})}(\theta_0)$. Write $z = (x-\mu)/\sigma$. Our choice of $\eps_0$ and the fact that $\abs{z_0} > 1 / \eps_0 > 2$ imply that $\abs{z/z_0 - 1} \le 1/2$; see Lemma~\ref{lem:relerr}. Consider three subcases: $\gamma > 0$, $\gamma = 0$, and $\gamma < 0$. 
\begin{itemize}
\item
Suppose $\gamma > 0$. By \eqref{eq:h0+}, we have $\omega(\theta) \le \mu_0 - (3/4) \sigma_0 (2\abs{z_0}) = \mu_0 + (3/2) \sigma_0 z_0 < \mu_0 + \sigma_0 z_0 = x$. Hence, $x \in S_\theta$. As $\gamma > 0$ and $z < 0$, the bounds in \eqref{eq:gamma0+-1} and \eqref{eq:gamma0+-2} then imply that $\| \dot \ell_\theta(x) \| \lesssim u^{1+2\eps_0}$. We need to bound $u$ from above. Recall that $u_\gamma(z)$ is decreasing in $z$ and increasing in $\gamma$. Moreover, $(3/2)z_0 \le z < 0$ and $0 < \gamma < 1 / (2\abs{z_0})$. We find 
\[ 
  1 
  \le u 
  \le \left(1 + \frac{1}{2\abs{z_0}} \frac{3}{2} z_0 \right)^{- 2\abs{z_0} }
  = \bigl(1 - (3/4) \bigr)^{-2\abs{z_0}}
  = 16^{-z_0}.
\]
For arbitrary $\rho > 1$, the function $z_0 \mapsto \rho^{-z_0}$ is integrable with respect to the standard Gumbel density $z_0 \mapsto \exp(-e^{-z_0}) \, e^{-z_0}$.
\item
Suppose $\gamma = 0$. The expressions for the score functions are $\partial_\gamma \ell_\theta(x) = (1 - e^{-z}) z^2/2 - z$, $\partial_\mu \ell_{\theta}(x) = (1 - e^{-z})/\sigma$, and $\partial_\sigma \ell_\theta(x) = ((1 - e^{-z})z - 1)/\sigma$. Since $(3/2)z_0 \le z \le (1/2)z_0 < 0$ for all $\theta$ considered, these expressions can be easily bounded by $P_{\theta_0}$-square integrable functions.
\item
Suppose $\gamma < 0$. Since $z < 0$, the components of the score vector can be bounded by a multiple of $u \, \max\{z, (\log u)^2\}$. But since $u_\gamma(z)$ is increasing in $\gamma$ and since $\gamma < 0$, we can bound $u$ by its value at $\gamma = 0$, which is $e^{-z}$. Now continue as for the case $\gamma = 0$.
\end{itemize}

\smallskip
\noindent
\textit{III.2 -- Case $\mu_0 - \sigma_0/\eps_0 \le x \le \mu_0 + \sigma_0/\eps_0$.}
This case is trivial, since the three components of the score vector $\dot{\ell}_\theta(x)$ are continuous and thus uniformly bounded on the closure of the bounded domain $\{ (x, \theta) : \abs{x - \mu_0} \le \sigma_0 / \eps_0, \, \theta \in \Theta_0(x) \}$.

\smallskip
\noindent
\textit{III.3 -- Case $x > \mu_0 + \sigma_0/\eps_0$.} 
This case is partially similar to the case $x < \mu_0 - \sigma_0 / \eps_0$. For $\theta \in \Theta_0(x) \subset U_{\eps_0}(\theta_0)$, we have $x > \mu_0 + \eps_0 > \mu$ and thus $z = (x-\mu)/\sigma > 0$ and $0 < u = u_\gamma(z) < 1$. Moreover, $z_0>1/\eps_0$.

We claim that $\Theta_0(x) \subset U_{1/(2z_0)}(\theta_0)$. Indeed, the point $\theta' = (-1/z_0, \mu_0, \sigma_0)$ is such that $\norm{\theta - \theta_0} = 1/z_0 = 2 / (2z_0)$ and still $x \notin S_{\theta'}$.

Let $\theta \in \Theta_0(x) \subset U_{1/(2z_0)}(\theta_0)$ and write $z = (x - \mu)/\sigma$. Again, we have $\abs{ z/z_0 - 1 } \le 1/2$. Consider three subcases: $\gamma > 0$, $\gamma = 0$, and $\gamma < 0$.
\begin{itemize}
\item
Suppose $\gamma > 0$. Since $0 < u < 1$ and $1 + \gamma z > 1$, all three components of the score vector can be bounded by a constant multiple of $z^2$ and thus by a constant multiple of $z_0^2$. Now it suffices to observe that all moments of the Gumbel distribution are finite.
\item
Suppose $\gamma = 0$. Then apply the same reasoning as for the subcase $\gamma = 0$ in the case III.1 above.
\item
Suppose $\gamma < 0$. By \eqref{eq:h0-}, we have $\omega(\theta) \ge \mu_0 + (3/4) \sigma_0 (2 z_0) = \mu_0 + (3/2) \sigma_0 z_0  \le \mu_0 + \sigma_0 z_0 = x$, so that $x \in S_{\theta}$. Since $0 > \gamma > - 1/(2z_0) \ge - 3/(4z) > - 1/z$, the bound in \eqref{eq:gamma-:z+} applies. We need to find a square-integrable bound on $z^2 / (1 + \gamma z)$. But this is immediate, since the denominator is larger than $1 - 3/4 = 1/4$ and the numerator is smaller than a multiple of $z_0^2$. \qedhere
\end{itemize}
\end{proof}

\section*{Acknowledgments}

The authors would like to thank two anonymous referees and an Associate Editor for their constructive comments on an earlier version of this manuscript.

The research by A.\ B\"ucher  has been supported by the Collaborative Research Center ``Statistical modeling of nonlinear dynamic processes'' (SFB 823, Project A7) of the German Research Foundation, which is gratefully acknowledged. 
Parts of this paper were written when A.\ B\"ucher was a visiting professor at TU Dortmund University.

J. Segers gratefully acknowledges funding by contract ``Projet d'Act\-ions de Re\-cher\-che Concert\'ees'' No.\ 12/17-045 of the ``Communaut\'e fran\c{c}aise de Belgique'' and by IAP research network Grant P7/06 of the Belgian government (Belgian Science Policy).

\bibliographystyle{chicago}
\bibliography{biblio}
\end{document}